\documentclass[11pt,a4paper,reqno]{amsart}

\usepackage{a4wide}

\usepackage{amssymb,amsmath,amsthm,mathrsfs}
\usepackage{graphicx}
\usepackage{float}
\usepackage{colortbl}
\usepackage{enumerate}
\usepackage{upref}
\usepackage{stackrel}
\usepackage{enumitem}

\usepackage{dsfont}

\usepackage[bookmarks=false]{hyperref}

\usepackage{accents}
    
\usepackage[colorinlistoftodos]{todonotes}

\usepackage{amsfonts} 
\DeclareMathSymbol{\shortminus}{\mathbin}{AMSa}{"39}

\usepackage{tikz}
\usetikzlibrary{arrows}
\usetikzlibrary{decorations.pathreplacing}

\usepackage{chngcntr}
\counterwithin{figure}{section}

\theoremstyle{plain}
\newtheorem{theorem}{Theorem}[section]
\newtheorem*{theorem*}{Theorem}
\newtheorem{lemma}[theorem]{Lemma}

\newtheorem{proposition}[theorem]{Proposition}
\newtheorem{corollary}[theorem]{Corollary}

\theoremstyle{definition}
\newtheorem{definition}[theorem]{Definition}

\newtheorem{example}[theorem]{Example}

\theoremstyle{remark}
\newtheorem{remark}[theorem]{Remark}

\newcommand{\eeq}{\end{equation}}
\newcommand{\beq}{\begin{equation}}

\newcommand{\vertiii}[1]{{\left\vert\kern-0.25ex
  \left\vert\kern-0.25ex\left\vert #1
  \right\vert\kern-0.25ex\right\vert\kern-0.25ex\right\vert}}

\def\R{\ensuremath{\mathbb R}}
\def\N{\ensuremath{\mathbb N}}

\def\ie{{\em i.e.}, }

\numberwithin{equation}{section}  

\begin{document}

\title[Distributional results for the shortest distance]{Distributional results for the shortest distance between
  trajectories of different dynamics}

\author[R. Aimino]{Romain Aimino}
\address{Romain Aimino\\ Centro de Matem\'{a}tica da Universidade do Porto\\ Rua do
Campo Alegre 687\\ 4169-007 Porto\\ Portugal}
\email{\href{mailto:romain.aimino@fc.up.pt}{romain.aimino@fc.up.pt}}
\urladdr{\url{https://sites.google.com/view/romain-aimino/accueil}}

\author[T. Caby]{Th\'eophile Caby}
\address{Th\'eophile Caby\\ Laboratoire d’Océanographie Physique et Spatiale, Univ Brest CNRS Ifremer IRD, Brest, France \\} \email{\href{mailto:theophile.caby@uni-brest.fr}{theophile.caby@uni-brest.fr}}

\author[J. M. Freitas]{Jorge Milhazes Freitas}
\address{Jorge Milhazes Freitas\\ Centro de Matem\'{a}tica \& Faculdade de Ci\^encias da Universidade do Porto\\ Rua do
Campo Alegre 687\\ 4169-007 Porto\\ Portugal}
\email[Corresponding author]{\href{mailto:jmfreita@fc.up.pt}{jmfreita@fc.up.pt}}
\urladdr{\url{http://www.fc.up.pt/pessoas/jmfreita/}}

\author[D. S. Pinho]{Duarte S\'a Pinho}
\address{Duarte S\'a Pinho\\ Centro de Matem\'{a}tica \& Faculdade de Ci\^encias da Universidade do Porto\\ Rua do
Campo Alegre 687\\ 4169-007 Porto\\ Portugal}
\email{\href{mailto:duarte.sa.pinho02@gmail.com}{duarte.sa.pinho02@gmail.com}}

\thanks{All authors were partially supported by CMUP, member of LASI, which is financed by national funds through FCT – Fundação para a Ciência e a Tecnologia, I.P., under the projects with reference UID/00144/2025 and associated DOI given by \href{https://doi.org/10.54499/UID/00144/2025}{https://doi.org/10.54499/UID/00144/2025}. DSP gratefully acknowledges the scholarship awarded by the Calouste Gulbenkian Foundation under the Novos Talentos programme, during which this project was initiated. TC was supported by ISblue project, Interdisciplinary graduate school for the blue planet  (ANR-17-EURE-0015) and co-funded by a grant from the French government under the program 'Investissements d'Avenir' embedded in France 2030.
}

\date{}



\begin{abstract}
We establish Extreme Value Distributions for the closest encounter
between trajectories generated by different maps defined in the same
reference phase space. For a class of strongly mixing maps, we show
that the limit distribution depends on the
length of the different trajectories and the co-dimension of the
associated invariant measures. It is also modulated by an Extremal
Index, that informs on the tendency of nearby points to diverge along
with the evolution of their respective dynamics, serving as an indicator of their compatibility. We give a formula
for this quantity for a class of chaotic maps of the interval and for
the co-dimension in the case when the respective measures admit
densities with isolated zeros and singularities. We present diverse
examples of systems satisfying these assumptions and compute the
different parameters modulating the limit distribution.
\end{abstract}

\maketitle

\section{Introduction}

The statistical behaviour of the shortest distance between orbits and
trajectories of dynamical systems has been studied for several decades,
beginning at least with the pioneering work of \cite{CC94} and its
subsequent developments in \cite{CCC09}. In recent years, however, the
subject has attracted renewed attention, partly due to its connection with
sequence matching in symbolic dynamics, observed in
\cite{short}. 

Distributional results for trajectories generated by the same dynamical
system were obtained in \cite{d2,dq} using techniques from extreme value
theory, while almost sure (a.s.) convergence results for the shortest distance
between two orbits appear in \cite{short}, and for multiple orbits in
\cite{mult}. These a.s. results were later generalised to
observed systems \cite{encoded,matchobs} and to slowly mixing systems
\cite{matchslow}. These a.s. developments were carried further in the 
 work \cite{asymp}, which considers two trajectories driven by
different maps and derives asymptotic results for the shortest
distance, while the observed counterpart for distinct dynamical systems was
recently investigated in \cite{barros-coutinho}.

Further developments include the study of shortest distances between partial
orbits \cite{liu-shi} and extensions to conformal iterated function systems
\cite{shi-ifs}. Very recent advances on the minimal distance between random
orbits were obtained in \cite{grs}.

In this work, we study multiple trajectories generated by possibly distinct
maps and establish full distributional results for their minimal mutual
distance. In this setting, an extreme event corresponds to a strong form of
synchronisation: two or more trajectories, starting from independent initial
conditions and evolving under different dynamics, approach each other
arbitrarily close at the same time. We derive the limiting distribution,
identify the associated modulating parameters — namely the co-dimension
$C_q$ and the Extremal Index $\theta$ — and provide several explicit
examples illustrating the theory.

The co-dimension plays an important role in the normalising sequences that assure distributional convergence. It essentially accounts for the irregularities of the invariant measure, namely, when zeros or singularities of the density arise with a significant impact or when its support is made out of fractal sets, for example. In the case of smooth invariant densities the co-dimension is 1.

The Extremal Index which changes the limiting law is seen here to serve as an indicator of the compatibility between the factor dynamics.  Essentially, it detects when the factor dynamics have a tendency to remain synchronised, once a synchronisation occurs. Typically, different dynamics have irrelevant compatibility, in the sense that most of the same initial conditions lead to different trajectories (in measure-theoretic terms), which is reflected in an Extremal Index equal to 1. This is because, in these extreme value approaches, as in \cite{d2,synchro}, we consider an observable function that takes maximal values when the orbits synchronise and then, since when the trajectories synchronise there is no strong propensity for the system to remain synchronised, then we have no clustering of high observations and the Extremal Index takes the value 1. By contrast, when the factor dynamics coincide, namely, when one studies trajectories generated by the same map, as in \cite{CC94,CCC09,d2,dq}, synchronisation persists indefinitely once it occurs. In particular, this means that when we have near-synchronisation, corresponding to the observation of a high value, there is a clear predisposition to remain close to synchronisation, which creates clustering of high observations, reflected in an Extremal Index less than 1. Then, as observed in \cite{CC94}, the Extremal Index can be interpreted as an average of the rate of expansion of the factor dynamics. In more physical terms, the Extremal Index reflects the positive Lyapunov exponent, in low dimensional dynamics and the metric entropy in higher dimensions, as insightfully pointed out in \cite{d2,ei}. This is because these quantities essentially express the `tendency' of the orbits to move away from each other.

In this work, we introduce examples in which the factor dynamics are compatible without being identical. More precisely, the dynamics coincide on part of the phase space and differ on its complement. In this setting, the Extremal Index becomes a weighted average of the two regimes: it captures the average expansion rate on the region where the dynamics agree, while taking the value 1 on the complementary region. The corresponding weights are given precisely by the relative measures of these two parts of the phase space.

We will use spectral methods of perturbed Perron-Frobenius operators introduced by Keller and Liverani in \cite{kl99,kl} to establish an abstract result which will enable us to apply their framework to the problem of obtaining distributional limits for the closest distance between trajectories generated by possibly distinct dynamics. Then, we will develop applications of this result to illustrate the roles of the co-dimension and, especially, the Extremal Index in the interpretation of the different regimes. 

The structure of the paper is as follows. In Section~\ref{sec:setting}, we establish the setting and formulate the problem to be studied. In Section~\ref{sec:EVL_conditions}, we revise Keller and Liverani's framework, introduce the notation and gather the main tools. In Section~\ref{subsec:mainthm}, we prove the abstract result that we will use subsequently to study the distributional limits mentioned earlier. In Section~\ref{sec:EI}, we provide general conditions that allow us to prove that the Extremal Index is less than 1, corresponding to the typical case, when the factor dynamics is incompatible, which means that we have unclustered synchronisation profiles. In Section~\ref{sec:compatibility-EI}, we show the usefulness of the Extremal Index as an indicator of the compatibility of the factor dynamics, using the developed tools to compute it for linear dynamical systems, with which we play to provide an example of distinct factor systems with a compatible component. In Section~\ref{sec:examples}, we provide further applications of the theory developed to handle factor dynamics such as perturbed linear maps. We end with an appendix 
where we discuss the possible role the co-dimension and its impact on the estimation of the quantities considered earlier.

\section{The setting}
\label{sec:setting}
Let $\mathcal{M}$ be a compact metric space equipped with a distance
$d$. For some $q\in \mathbb{N}$, $q\ge 2$, we consider the dynamical
systems $(\mathcal{M},T_1,\mu_1)$,
$(\mathcal{M},T_2,\mu_2),\dots,(\mathcal{M},T_q,\mu_q)$, where for
each $i=1,\dots,q$, $T_i:\mathcal{M} \to \mathcal{M}$ is a discrete
transformation that leaves the probability measure $\mu_i$ invariant.
We will denote $\overline{\mu}:=\mu_1 \times \mu_2 \times \cdots
\times \mu_q$ the product measure with support in $\mathcal{M}^q$,
and $\overline{T}:\mathcal{M}^q\to \mathcal{M}^q$ the product map
\begin{equation}
\label{eq:def-product-system}
\overline{T}(\overline{x}):=(T_1(x_1),\ldots,T_q(x_q)),
\end{equation}
where $\overline{x}=(x_1,\ldots,x_q)\in \mathcal{M}^q$.

\begin{definition}\label{defdq}
We call the \emph{co-dimension} of the measures $\mu_1,\ldots,\mu_q$
the following quantity (if it exists):
\begin{equation}
C_q(\mu_1,\ldots,\mu_q):=\lim_{r\to 0}
\frac{\log \int_{\mathcal{M}} \prod_{i=2}^q\mu_i(B_r(x))\,d\mu_1(x)}
{(q-1)\log r},
\end{equation}
where $B_r(x)$ denotes the ball of radius $r$ centred at
$x\in \mathcal{M}$.
\end{definition}

\begin{remark}
By Fubini, we have that $C_q$ is symmetric in $\mu_1,\ldots,\mu_q$. 
Moreover, when $\mu_1=\mu_2=\cdots=\mu_q$,
 then  $C_q(\mu_1,\ldots,\mu_q)=D_q(\mu_1)$, which stands for
  the generalised Rényi dimension of order $q$ of $\mu_1$.
\end{remark}
We define the observable function 
\begin{align}
\varphi\colon [0,1]^q& \longrightarrow \R\cup\{\infty\}\nonumber\\ 
\left(x_1,\ldots,x_q\right)&\longmapsto -\log\max_{j\in\left\{2,\ldots,q\right\}} d\left(x_1,x_j\right),
 \label{eq:def-observable}
\end{align}
which will serve as a synchronisation indicator. We define now the process $(Y_i)_{i}$ simply by evaluating $\varphi$ along the orbits of the  system:
$$Y_i(\overline{x}):=\varphi \circ \bar T (\bar x)=-\log\max_{j=2,\dots,q}d(T_1^i (x_1),T_j^i (x_j)),$$
where $\overline{x}\in\mathcal{M}^q$ is drawn from $\overline{\mu}$.

Note that $Y_i=\infty$ means that the system is synchronised at time $i$ as all factors are in the same state, \ie $T_1^i(x_1)=T_2^i(x_2)=\ldots=T_q^i(x_q)$.

As usual, in order to study the Extreme Value properties of this process, we take a sequence of thresholds
$(u_n(s))_{n}$, $s>0$, satisfying
\begin{equation}\label{tau}
\overline{\mu}(Y_0 > u_n(s)) \sim \frac{s}{n}.
\end{equation}

Since the $q$ trajectories are independent (chosen with respect to the product measure),
\begin{equation}\label{dqf}
\overline{\mu}(Y_0> u_n(s))
  =\int_{\mathcal{M}} \prod_{i=2}^q\mu_i(B_{r_n}(x))\,d\mu_1(x)
  \sim e^{-u_n(s)C_q(q-1)},
\end{equation}
where $r_n=e^{-u_n}$. Matching \eqref{tau} and \eqref{dqf} gives
$$u_n(s)=\frac{\log n}{C_q(q-1)}-\frac{\log s}{C_q(q-1)}.$$

The quantity $\overline{\mu}(Y_0>u_n)$ corresponds to the probability that all
points $x_1,\ldots,x_q$ lie in the ball $B_{r_n}(x_1)$, \ie that
the product dynamics enters the target set
\begin{equation}
\label{eq:def-Deltan}
\Delta_n=\left\{\overline{x}\in \mathcal{M}^q :
  \max_{j=2,\dots,q}d(x_1,x_j) < r_n\right\}.
\end{equation}
Following \cite{fft}, we define 
\begin{equation}
\label{eq:def-Mn}
M_n:=\max\{Y_0,\dots,Y_{n-1}\}.
\end{equation}  
Observe that the event $\{M_n\leq u_n\}$ corresponds to the 
dynamics never entering $\Delta_n$, during the first $n$ steps, which means that the systems never came close to being synchronised. Also note that the value of $M_n$ is related to the shortest distance reached by $q$ trajectories up to time $n$.

\section{A spectral approach to study Extreme Value Laws}
\label{sec:EVL_conditions}

In \cite{kl99}, Keller and Liverani introduced a powerful framework of spectral perturbation methods, which, later, they developed further and applied to the study Rare Events and Escape Rates in \cite{kl}. This framework has since then been exploited to study Extreme Values in \cite{k,AFV15, AFGV25, AHV25}, for example. We will use these spectral perturbation techniques to obtain our main results and, therefore, we start by describing them and recalling their main features.

\subsection{Abstract perturbation results}
\label{subsec:KL}

Let $\left(V, \| \cdot\|\right)$ be a real or complex normed vector space with dual $\left(V^*, \|\cdot \|\right)$.
Suppose that we have a family of operators $\left(P_\varepsilon\right)_{\varepsilon\in E}$, satisfying
\begin{equation} \label{A0}
  \|P_\varepsilon\|\leq M, \tag{A0}
\end{equation}
where $E\subseteq \mathbb{R}$ is a closed set of parameters having $\varepsilon=0$ as an accumulation point.
Suppose that there are $\lambda_\varepsilon \in \mathbb{C}, \, \varphi_\varepsilon \in V, \, \nu_\varepsilon \in V^*$ and linear operators $Q_\varepsilon:V \to V$ such that
\begin{equation} \label{A1}
  \lambda_\varepsilon^{-1} P_\varepsilon=\varphi_\varepsilon \otimes \nu_\varepsilon+ Q_\varepsilon \tag{A1}
\end{equation}
\begin{equation}\label{A2}
  P_\varepsilon\left(\varphi_\varepsilon\right)=\lambda_\varepsilon \varphi_\varepsilon, \, \nu_\varepsilon P_\varepsilon =\lambda_\varepsilon \nu_\varepsilon, \, Q_\varepsilon\left( \varphi_\varepsilon\right)=0, \,\nu_\varepsilon Q_\varepsilon=0 \tag{A2}
\end{equation}
\begin{equation} \label{A3}
  \sum_{n=0}^\infty \sup_{\varepsilon\in E} \|Q_\varepsilon^n\|=:C_1 <\infty. \tag{A3}
\end{equation}
Conditions \eqref{A1}+\eqref{A2} imply that $\nu_\varepsilon\left(\varphi_\varepsilon\right)=1,$ for all $\varepsilon\in E$.
We also control the size of $\varphi_\varepsilon$ relatively to the size of $\varphi_0$.
\begin{equation} \label{A4}
  \nu_0\left(\varphi_\varepsilon\right)=1 \text{ and } \sup_{\varepsilon\in E} \|\varphi_\varepsilon\|=:C_2<\infty. \tag{A4}
\end{equation}
Define
\begin{equation} \label{delta_epsilon}
  \delta_\varepsilon:=\nu_0\left(\left(P_0-P_\varepsilon\right)\left(\varphi_0\right) \right)
\end{equation}
and suppose that there exists $C_3>0$ such that
\begin{equation} \label{A5}
  \eta_\varepsilon:=\left\|\nu_0\left(P_0-P_\varepsilon\right)\right\| \to 0 \text{ as } \varepsilon\to 0 \tag{A5}
\end{equation}
\begin{equation} \label{A6}
  \eta_\varepsilon \cdot \left\|\left(P_0-P_\varepsilon\right)\left(\varphi_0\right)\right\|\leq C_3 \left| \delta_\varepsilon\right|. \tag{A6}
\end{equation}
The main abstract result in \cite[(Theorem~2.1)]{kl} tells us that under 
  assumptions \eqref{A0}-\eqref{A6}, then, one of the following holds
  \begin{enumerate}[label=(\alph*)]
    \item There is $\varepsilon_0>0$ such that $\lambda_\varepsilon=\lambda_0$ if $\varepsilon \leq \varepsilon_0$ and $\delta_\varepsilon=0$.
    \item If $\delta_\varepsilon\neq0$ for all sufficiently small $\varepsilon \in E$ and if
    \begin{equation} \label{A7}
      p_k=:\lim_{\varepsilon\to 0} p_{k, \varepsilon}:=\lim_{\varepsilon\to 0} \frac{\nu_0\left(\left(P_0-P_\varepsilon \right)P_\varepsilon^k \left(P_0-P_\varepsilon\right)\left(\varphi_0\right)\right)}{\delta_\varepsilon} \tag{A7}
    \end{equation}
    exists for each integer $k\ge0$, then
    \begin{equation} \label{eqn:expansion}
      \lim_{\varepsilon\to0} \frac{\lambda_0-\lambda_\varepsilon}{\delta_\varepsilon}=1-\sum_{k=0}^\infty \lambda_0^{-(k+1)}p_k.
    \end{equation}
  \end{enumerate}

Moreover, as seen in \cite{kl99}, the following conditions are sufficient to verify \eqref{A0}-\eqref{A4}:
\begin{enumerate}[label=(B\arabic*)]
  \item \label{B1} The operator $P_0$ has the form
  \begin{equation} \label{B1.1 form of P_0}
    P_0=\varphi_0\otimes\nu_0+Q_0\text{, where } \nu_0\left(\varphi_0\right)=1 \text{, } \nu_0 Q_0=0 \text{, }Q_0\varphi_0=0, \tag{B1.1}
  \end{equation}
  and the spectral radius of $Q_0$ is strictly less than 1.
  \item \label{B2}There exists $r\in \left(0,1\right)$, $D>0$ and a semi-norm $\left| \cdot \right|_w \leq \|\cdot\|$ in $V$ such that
  \begin{equation} \label{B2.1 res spec}
    \text{The residual spectrum of $P_\varepsilon$ is contained in $\left \{z\in\mathbb{C}:|z|\leq r \right\}$}. \tag{B2.1}
  \end{equation}
  \begin{equation} \label{B2.2 P_n uniformly bounded in w norm}
    \forall \varepsilon\in E, \, \forall f \in V, \, \forall n \in \mathbb{N}, \, \left| P^n_\varepsilon f \right|_w \leq D \left| f\right|_w \tag{B2.2}
  \end{equation}
  \begin{equation} \label{B2.3 Condition for P^n ||,|| norm}
    \forall \varepsilon \in E, \, \forall f \in V, \, \forall n \in \mathbb{N}, \, \|P^n_\varepsilon f\| \leq Dr^n \|f \|+D \left|f\right|_w \tag{B2.3}
  \end{equation}
  \item \label{B3} There exists a decreasing upper-semicontinous function $\pi_\varepsilon:E \to (0, \infty)$ such that
  \begin{equation} \label{B3.1 continuity_triple_norm}
    \lim_{\varepsilon\to 0} \pi_\varepsilon=0 \text{ and } \forall f \in V, \space \forall \varepsilon \in E: \left|P_\varepsilon f -P_0 f \right|_w\leq \pi_\varepsilon \|f\|. \tag{B3.1}
  \end{equation}
\end{enumerate}
\begin{remark}
  The following notation will be useful.
  Let $R: V\to V$ be a linear operator. We define
  \begin{equation}
    \vertiii{R}:=\sup \left\{\left| Rf\right|_w: f\in V, \,\|f\|\leq 1 \right\}.
  \end{equation}
  With this notation, we can write the last part of condition \ref{B3} as $\vertiii{P_\varepsilon-P_0}\leq \pi_\varepsilon$.
\end{remark}

\subsection{Extreme Values from a spectral approach}
\label{subsec:REPFO}
Following \cite{k}, we consider the so called Rare event Perron–Frobenius operators (REPFO) setting. Let $[0,1]^q$ be the phase space, and $\mathrm{Leb}$ the reference measure defined on the Borel subsets of $[0,1]^q$.
Assume that $V$ is a Banach space of functions continuously embedded and dense in $L^1\left([0,1]^q, \mathrm{Leb}\right)$ which contains the constant functions and $P_0$ is the transfer operator of a non-singular map $T : [0,1]^q \to [0,1]^q$ with respect to the reference measure $\mathrm{Leb}$.
Recall that in this setting, $\mathrm{Leb}$ defines a functional on $V$, by associating to each $f \in V$ the number $\int f \, d\mathrm{Leb}$. We assume that $\nu_0=\mathrm{Leb}$.
We also assume that there exists a family $\left(A_n\right)_{n\in\mathbb{N}}$ of sufficiently regular subsets of $[0,1]^q$ such that
\begin{enumerate}[label=(R\arabic*)]
  \item \label{R1}For all $f\in V$, $\mathds{1}_{A_n}\cdot f\in V$.
  \item \label{R2} The operators $P_n$, defined by $P_n f=P_0\left(\mathds{1}_{[0,1]^q\setminus A_n}f \right)$ satisfy assumptions \ref{B1}-\ref{B3}, with $\left|\cdot \right|_w=\|\cdot\|_{L^1\left([0,1]^q, \mathrm{Leb}\right)}$.
  \item \label{R3} There exists $C>0$ such that for all $f\in V$, $\left|\nu_0\left(\mathds{1}_{A_n}f\right)\right| \cdot \|\mathds{1}_{A_n} \varphi_0\| \leq C \|f\| \cdot \left|\nu_0\left(\mathds{1}_{A_n} \varphi_0 \right) \right|$.
\end{enumerate}
We note that since \ref{B1} is satisfied, $\mu:= \varphi_0\,\mathrm{Leb}$ is the unique invariant measure that is absolutely continuous with respect to the reference measure $\mathrm{Leb}$. In what follows, we will also denote $h = \varphi_0$.

\begin{lemma} \label{lemma: critera for R3}
  Assume that $V$ is a space of functions on $[0,1]^q$ that is a Banach algebra and is continuously embedded in $L^\infty\left([0,1]^q, \mathrm{Leb}\right)$.
  Suppose that $\inf_{x\in \bigcup_{n}A_n} h>0$.
  In addition, suppose that $\|\mathds{1}_{A_n} \|$ is uniformly bounded.
  Then \ref{R3} is satisfied.
\end{lemma}
\begin{proof}
  Let $f \in V$. Since $V\subseteq L^\infty\left([0,1]^q, \mathrm{Leb}\right)$,
  \begin{equation*}
    \left|\nu_0\left(\mathds{1}_{A_n}f\right)\right|=\left|\int\mathds{1}_{A_n}f\, d\mathrm{Leb} \right| \leq \|f\|_{L^\infty}\cdot \left|\int\mathds{1}_{A_n}\, d\mathrm{Leb} \right|.
  \end{equation*}
  Using the fact that the injection of $V$ in $L^\infty$ is continuous, we get
  \begin{equation*}
    \left|\nu_0\left(\mathds{1}_{A_n}f\right)\right|\leq C_1 \|f\| \cdot \left|\int\mathds{1}_{A_n}\, d\mathrm{Leb}\right|.
  \end{equation*}
  The assumption that $\inf_{x\in \bigcup_{n}A_n} h>0$ allows us to evaluate $\mathds{1}_{A_n}$ with respect to the measure $\mu:=h\,\mathrm{Leb}$:
  \begin{equation*}
    \left|\int\mathds{1}_{A_n}\, d\mathrm{Leb}\right|\leq \frac{1}{\inf_{x\in \bigcup_{n}A_n} h}\left|\int\mathds{1}_{A_n} h\, d\mathrm{Leb}\right|.
  \end{equation*}
  Since $\left(V, \|\cdot \|\right)$ is a Banach algebra and $\|\mathds{1}_{A_n} \|$ is uniformly bounded, we conclude that
  \begin{equation*}
    \|\mathds{1}_{A_n}h\| \leq C_2'\|\mathds{1}_{A_n}\|\cdot\|h\| \leq C_2.
  \end{equation*}
  Putting all of this together,
  \begin{equation*}
    \left|\nu_0\left(\mathds{1}_{A_n}f\right)\right|\cdot \|\mathds{1}_{A_n}h\|\leq \frac{C_1C_2 \|f\|}{\inf_{x\in \bigcup_{n}A_n} h}\left|\int\mathds{1}_{A_n} h\, d\mathrm{Leb}\right| \leq C \|f\|\cdot \left|\nu_0\left(\mathds{1}_{A_n}h \right) \right|,
  \end{equation*}
  which concludes the proof.
\end{proof}

Now,  we illustrate how one can establish the existence of a limiting EVL in a REPFO setting.
According to \cite[Proposition 1]{k}, the operators $\left(P_n\right)_n$ satisfy conditions \ref{B1}-\ref{B3}, so 
$$P_n=\varphi_n\otimes \nu_n+Q_n\quad \text{and}\quad \sup_{n}\|Q_n^k\|\leq K_3 r^k,\quad \text{where}\quad 0<r<1.$$
Letting $\delta_n:=\int\mathds{1}_{A_n} h \,d\mathrm{Leb}=\mu\left(A_n\right)\neq 0$, for all sufficiently large $n$, we have by \eqref{eqn:expansion} the expansion:
\begin{equation}
  \lim_{n\to\infty}\frac{1-\lambda_n}{\delta_n}=\theta :=1-\sum_{j=0}^\infty p_j ,
\end{equation}
if we assume now that for each $k\in \mathbb{N}$, the following limit exists:
\begin{equation} \label{definition:p_k}
  p_k=\lim_{n\to \infty} p_{k,n} := \frac{\mu\left(A_n\cap T^{-1}A_n^c \cap\cdots\cap T^{-k}A_n^c \cap T^{-(k+1)}A_n\right)}{\mu\left(A_n\right)}.
\end{equation}

The first hitting time to the set $A_n$ is $\tau_n(x)=\inf\{i \ge 0:T^i x \in A_n\}$. Since $\nu_0=\mathrm{Leb}$ and that $\nu_0\left(\varphi_n\right)=1$, we have
\begin{align*}
  \mu\left(\left\{\tau_n\ge k\right\}\right)=\int_{\left\{\tau_n\ge k\right\}}h \,d\mathrm{Leb}&= \int\prod_{i=0}^{k-1} \left(\mathds{1}_{[0,1]^q\setminus A_n}\circ T^i\right)h\,d\mathrm{Leb}\\
  &=\int P_n^k h\, d\mathrm{Leb}\\
  &=\lambda_n^k\nu_n\left(h\right)+\mathcal{O}\left(\lambda_n^k\|Q^k_n\|\cdot\|h\|\right).
\end{align*}
The quantity $\nu_n\left(h\right)\to 1$ as $n\to\infty$.
Suppose that $n\mu\left(A_n\right)\to s$ for some $s>0$. Then,
\begin{equation}
  \lim_{n\to \infty}\mu\left(\left\{ \tau_n\ge n\right\} \right)=\lim_{n\to\infty}\exp\left(-n\mu\left(A_n\right)\theta+o\left(\mu\left(A_n\right)\right)\cdot n\right)=e^{-\theta s}.
\end{equation}
In particular, returning to our original synchronisation problem described in Section \ref{sec:setting}, if $\varphi:[0,1]^q\to \mathbb{R}\cup \left\{\infty\right\}$ is given as in \eqref{eq:def-observable} and there exists a sequence $u_n$ such that $n\mu\left(\left\{\varphi>u_n\right\}\right)\to s$ and if the sets $\Delta_n=\left\{\varphi>u_n\right\}$ satisfy \ref{R1}-\ref{R3}, then
\begin{equation*}
  \lim_{n\to\infty}\mu\left( M_n \leq u_n \right)=e^{-\theta s}.
\end{equation*}

\section{Abstract limiting results for the product of interval maps}
\label{subsec:mainthm}

In this section we prove the existence of an extreme value law for
the process $(Y_i)$ of Section~\ref{sec:EVL_conditions} using the spectral
approach. We consider $\mathcal{M}=[0,1]^q$ and denote by $\mu_i=h_i\,\mathrm{Leb}$
the absolutely continuous invariant probability measures. We assume that the product map
$\tilde{T}:[0,1]^q\to[0,1]^q$ is piecewise smooth with branches
indexed by $j$ to distinguish from the indices $i=1,\ldots,q$ of the factors; we write $T_{j,i}$ for the restriction of the factor $T_i$ of $\tilde T$ to the domain of injectivity of the $j$-th branch of $\tilde T$. For this kind of maps the transfer operator can be expressed as 
$$
P(f)(\mathbf{x})=\sum_{\tilde T(\mathbf{y})=\mathbf{x}}\frac{f(\mathbf{y})}{\operatorname{Jac}\tilde T(\mathbf{y})}.
$$

\subsection{Quasi-H\"{o}lder space and Lasota-Yorke inequality}\label{subsec:valpha}

Before stating our main result, we need to introduce the functional space where we will work on. It was introduced by Saussol in \cite{saussol}. We define the quantity
\begin{equation} \label{Equation: V_alpha norm}
  \left|f\right|_\alpha=\sup_{0<\epsilon\leq \epsilon_0} \epsilon^{-\alpha}\int_{[0,1]^q} \operatorname{osc}\left(f, B_\epsilon\left(x\right)\right)\, dx.
\end{equation}
and the $\alpha$-Quasi-H\"older space by
\begin{equation} \label{Equation: V_alpha space}
  V_\alpha=\{f \in L^1([0,1]^q, \mathrm{Leb}):\left|f\right|_\alpha<\infty\}.
\end{equation}
The norm considered in $V_\alpha$ is $||\cdot||_\alpha=||\cdot||_{L^1\left([0,1]^q, \mathrm{Leb}\right)}+|\cdot|_\alpha$. As shown in \cite{saussol}, $V_\alpha$ is a Banach space, compactly embedded in $L^1$, continuously embedded in $L^\infty$, and is a Banach algebra. The space $V_\alpha$is independent of the choice of the parameter $\epsilon_0$, and different choices of $\epsilon_0$ yield equivalent norms.

We will assume that the transfer operator of the map $\tilde{T}$ satisfies a Lasota-Yorke inequality for the first iteration:

\begin{enumerate}
\item[(LY)] \text{There exist $\eta \in (0,1)$ and $D >0$ such that, for all $f \in V_\alpha$,}
\begin{equation}
\tag{LY}\label{cond:LY}
\|Pf\|_\alpha \le \eta \|f\|_\alpha + D \|f\|_{L^1}.
\end{equation}
\end{enumerate}

\subsection{Synchronisation distributional limits}
We now state and prove the main result of this section.
Consider the set $[0,1]^q$, where $q$ is an integer greater or equal to 2.
Let $\mathrm{Leb}$ be the reference measure. Consider the product dynamical system $\bar T\colon [0,1]^q \to [0,1]^q$ and the synchronisation observable function $\varphi\colon[0,1]^q \to \mathbb{R}\cup \left\{\infty\right\}$ given in \eqref{eq:def-product-system} and \eqref{eq:def-observable}, respectively.

\begin{theorem} \label{Thrm: Main result using spectral approach}
   Assume that the transfer operator of $\bar T$ satisfies conditions \ref{cond:LY} and \ref{B1} and that $h$ is bounded away from zero, \ie there exists $c>0$ such that $h\ge c$. Let $\mu:=h\,\mathrm{Leb}$. Assume that there exists a sequence $\left(u_n\right)_{n}$ such that
  \begin{equation} \label{equation: un in thrm spec approach}
    \lim_{n\to\infty} n\mu\left(\Delta_n\right)=\lim_{n\to\infty} n\mu\left(\left\{\varphi>u_n\right\}\right)=s,
  \end{equation}
  for some $s>0$, and, moreover, assume that for each $k\ge 0$, the limit
  \begin{equation} \label{equation: pk}
    p_k:=\lim_{n\to\infty} \frac{\mu\left(\Delta_n\cap \bar T^{-1}\Delta_n^c \cap\cdots\cap \bar T^{-k}\Delta_n^c \cap \bar T^{-(k+1)}\Delta_n\right)}{\mu\left(\Delta_n\right)}
  \end{equation}
  exists.
  Then, for all $s>0$, we have
  \begin{equation} \label{equation: final EVL spec ap thrm}
    \lim_{n\to\infty}\mu\left(M_n\leq u_n\right)=e^{-\theta s},
  \end{equation}
  where $\theta:=1-\sum_{k=0}^\infty p_k$.
\end{theorem}

\begin{remark}
  This theorem guarantees the existence of an EVL, once we verify that a sequence $(u_n)_n$ satisfying \eqref{equation: un in thrm spec approach} exists and that the limits $p_k$ exist. These conditions will be addressed in Sections~\ref{sec:EI}, \ref{sec:examples} and \ref{sec:compatibility-EI}.
\end{remark}

\begin{proof}
From discussion in the previous section, in order to prove Theorem~\ref{Thrm: Main result using spectral approach}, it suffices to check conditions \ref{R1}-\ref{R3} for the space $\left(V_\alpha, \|\cdot\|_\alpha\right)$.
The indicator function $\mathds{1}_{\Delta_n}$ over the strip along the diagonal, $\Delta_n = \{\varphi > u_n\}$, belongs to $V_\alpha$: its boundary has $(q-1)$-dimensional Lebesgue measure $O(r_n)$ (where $r_n = e^{-u_n}$), so $|\mathds{1}_{\Delta_n}|_\alpha = O(1)$. This together with the fact that $V_\alpha$ is a Banach algebra implies \ref{R1}. It is clear that all the assumptions of Lemma \ref{lemma: critera for R3} are satisfied, so \ref{R3} holds.
It remains to check condition \ref{R2}.
Starting by condition \ref{B3}, we see that
\begin{align*}
  \left\|P_n f -P_0 f \right\|_{L^1}\leq \left\|P_0\left(\mathds{1}_{\Delta_n}f \right) \right\|_{L^1}= \left\|\mathds{1}_{\Delta_n}f\right\|_{L^1} \\ \leq \|f\|_{L^\infty} \cdot \|\mathds{1}_{\Delta_n}\|_{L^1} \leq C \mu\left(\Delta_n\right)\|f\|_\alpha.
\end{align*}
This means that condition \ref{B3} is satisfied, with $\pi_n=C\mu\left(\Delta_n\right)$.
Condition \eqref{B2.2 P_n uniformly bounded in w norm} follows from the fact that
\begin{equation*}
  \|P_n^kf\|_{L^1}=\|P_0^k\left(\mathds{1}_{\Delta_n^c}f\right)\|_{L^1}=\|\mathds{1}_{\Delta_n^c}f\|_{L^1} \leq \|f\|_{L^1}.
\end{equation*}
To prove condition \eqref{B2.3 Condition for P^n ||,|| norm}, we note that, using \ref{cond:LY}:
\begin{align*}
  \|P_n f\|_\alpha=\|P_0\left(\mathds{1}_{\Delta_n^c}f\right)\|_\alpha &\leq \eta \| \mathds{1}_{\Delta_n^c}f\|_\alpha + D \|\mathds{1}_{\Delta_n^c}f\|_{L^1} \\
  &\leq\eta| \mathds{1}_{\Delta_n^c}f |_\alpha + (1+ D) \|f\|_{L^1}.
\end{align*}
To estimate $| \mathds{1}_{\Delta_n^c}f |_\alpha$ we proceed as in \cite[p3326]{synchro} and we obtain 
\[
| \mathds{1}_{\Delta_n^c}f |_\alpha \le |f|_\alpha ( 1 + C_q r_n),
\] 
where $C_q$ depends only on the dimension $q$ and the parameter $\epsilon_0$, and $r_n = e^{-u_n}$.
Condition \eqref{B2.3 Condition for P^n ||,|| norm} then follows for all large $n \ge n_0$ by setting $r = \eta (   1 + C_q r_{n_0}) <1$ and by iterating the inequality for $P_n^k$.
Condition \eqref{B2.1 res spec} follows from the fact that the unit ball of $(V_\alpha, \|\cdot\|_\alpha)$ is compact in the $L^1$-norm. \end{proof}

\subsection{Application to the product of piecewise expanding maps}
\label{subsec:piecewise-expanding-example}

In practice, to check the assumptions of Theorem \ref{Thrm: Main result using spectral approach}, we will need several results from Saussol \cite{saussol} that we recall below.

We say that a map $\tilde T:[0,1]^q \to [0,1]^q$ is piecewise expanding if there exists an at most countable family of disjoint open sets $U_j \subseteq [0,1]^q$ and $V_j$ such that $\overline{U_j}\subseteq V_j$ and maps $\tilde T_j:V_j \to \mathbb{R}^n$ satisfying, for some $0<\alpha\leq1$ and $\epsilon_0>0$, the following properties:
\begin{enumerate}[label=(PE\arabic*)]
  \item\label{condition: PE1} $\tilde T_j|_{U_j}=\tilde T|_{U_j}$.
  \item\label{condition: PE2} For all $j$, $\tilde T_j \in C^1(V_j)$, $\tilde T_j$ is injective and $\tilde T_j^{-1}\in C^1(\tilde T_j(V_j))$. Moreover, the determinant is uniformly H\"older: for all $i, \epsilon \leq \epsilon_0, z \in \tilde T_j (V_j)$ and $x,y\in B_\epsilon(z)\cap \tilde T_j V_j$,
  \begin{equation} \label{equation: PE2}
    \left|\operatorname{det}D_x\tilde T_j^{-1}-\operatorname{det}D_y\tilde T_j^{-1}\right|\leq c \left|\operatorname{det}D_z\tilde T_j^{-1}\right|\epsilon^\alpha.
  \end{equation}
  \item\label{condition: PE3} $\mathrm{Leb}\left([0,1]^q\setminus \bigcup_j U_j\right)=0$.
  \item\label{condition: PE4} There exists $\sigma<1$ such that for all $u,v \in \tilde T(V_j)$, with $d(u,v)\leq\epsilon_0$, we have $d(T_j^{-1}u, T_j^{-1}v)\leq \sigma\, d(u,v)$.
  \item\label{condition: PE5} Define $G(\epsilon,\epsilon_0):=\sup_{x\in [0,1]^q}G(x,\epsilon,\epsilon_0)$, where
  $$G(x,\epsilon,\epsilon_0):=\sum_j\frac{\mathrm{Leb}\left(T_j^{-1}B_\epsilon\left(\partial T(U_j)\right) \cap B_{(1-s)\epsilon_0}\left(x\right) \right)}{\mathrm{Leb}\left(B_{(1-s)\epsilon_0}\left(x\right)\right)}.$$
  The quantity $\eta(\epsilon_0):=\sigma^\alpha+2\sup_{\epsilon\leq \epsilon_0} \frac{G(\epsilon)}{\epsilon^\alpha}\epsilon_0^\alpha$ is such that $\sup_{\delta\leq \epsilon_0}\eta(\delta)<1$.
\end{enumerate}

In order to check condition \ref{cond:LY}, we have the following lemma.

\begin{lemma}[{\cite[Lemma 4.1]{saussol}}] \label{lemma: Lasota_Yorke in Valpha}
  Let $P$ be the transfer operator associated with a map $\tilde T$ satisfying \ref{condition: PE1}-\ref{condition: PE5}. Then, provided $\epsilon_0$ is small enough,  there exist $\eta \in (0,1)$ and $D <\infty$ such that, for all $f\in V_\alpha$, we have that $Pf\in V_\alpha$, with
  \begin{equation} \label{equation: LY for V_alpha}
    \left|Pf\right|_\alpha\leq\eta \left|f\right|_\alpha+D\|f\|_{L^1\left([0,1]^q, \mathrm{Leb}\right)}.
  \end{equation}
\end{lemma}

A product map  $\tilde T:[0,1]^q \to [0,1]^q$ such that each factor $T_i : [0,1] \to [0,1]$  is piecewise $C^{1+\alpha}$ with finitely many monotonicity intervals and uniformly expanding, i.e. $\inf |T_i'|>1$, clearly satisfies conditions~\ref{condition: PE1}--\ref{condition: PE4}.

Regarding condition \ref{condition: PE5}, \cite[Lemma 2.1]{saussol} shows that the condition is satisfied whenever
\[
\eta := \sigma^\alpha + \frac{4\sigma}{1-\sigma}\,
Y\,\frac{\gamma_{q-1}}{\gamma_q} < 1,
\]
provided that the dynamical partition is finite and its boundaries are contained in piecewise $C^1$ embedded compact codimension-one submanifolds. This applies in particular to the present product-map setting. Here, $\gamma_q$ denotes the volume of the unit ball in $\mathbb{R}^q$, and $Y$ is the maximal number of smooth boundary components of the domains of injectivity intersecting at a single point. In the present setting, we have $Y=2q$.

\begin{remark} In the case where all the factor maps $T_i$ are piecewise onto, it is not necessary to take the discontinuity into consideration, and we can directly set $\eta = \sigma^\alpha$, see for example the computations in \cite{ANT17}.
\end{remark}

According to \cite[Theorem 5.1]{saussol}, if the map $\tilde T$ satisfies conditions~\ref{condition: PE1}--\ref{condition: PE5}, its transfer operator is quasi-compact on $V_\alpha$, and the map $\tilde T$ admits a spectral decomposition. In order to guarantee that there exists a unique ergodic and mixing component, and then to guarantee condition \ref{B1}, we will use the following condition (adapted from \cite[Proposition 2.9]{annealed}) which imposes some sort of topological exactness, and also ensures that the unique invariant density if bounded uniformly away from 0.

\begin{proposition} \label{proposition: density_bounded_from_zero}
  Suppose that the partition $U_j$ associated to the map $\tilde T$ is finite. Assume that for any ball $B$ of $[0,1]^q$, there exists $n\in\mathbb{N}$ such that $\tilde T^nB=[0,1]^q$ (up to $\mathrm{Leb}$-measure zero), then there exists a unique absolutely continuous invariant probability (acip) $h\in V_\alpha$ for $\tilde T$. Furthermore, the support of $h$ is $[0,1]^q$ and $h$ is bounded away from zero.
\end{proposition}
\begin{proof}
  For $h$ an invariant density in $V_\alpha$, take $B$ to be the ball described in Lemma \cite[Lemma 3.1]{saussol}. Let $c$ be the infimum of $h|_B$. Then, for $\mathrm{Leb}$-almost every $x\in [0,1]^q$,
  \begin{equation*}
    h(x)=P_0^n h(x) \ge c P_0^n\left( \mathds{1}_B\right)\left(x\right)=c \sum_j \mathds{1}_{\tilde T_j^nB}\left(x\right) \left|\operatorname{det} D_{x}\left(\tilde T_j^n\right)^{-1}\right| \mathds{1}_{\tilde T^n_j U_{i,n}}(x).
  \end{equation*}
  Here, $U_{i,n}$ denotes the partition associated to $T^n$.
  Since $\tilde T^nB=[0,1]^q$, there must exist some $i$ such that $x\in \tilde T_j^n B$.
  This means that $\left(\tilde T_j^n\right)^{-1}\left(x\right) \in B \cap U_{i,n}$, so $x \in \tilde T_j^n (U_{i,n})$.
  Therefore,
  \begin{equation}
    h\ge c \inf_j \inf_x \left|\operatorname{det}D_x \left(\tilde T_j^n\right)^{-1}\right|>0,
  \end{equation}
  which concludes the proof.
\end{proof}

\section{Unclustered synchronisation profiles}
\label{sec:EI}

Throughout this section, we assume that $q=2$, so $\mathcal{M}=[0,1]^2$. The product map $\overline{T}$ has smooth branches denoted by $\bar T_j:[0,1]^2\to[0,1]^2$, and $T_{j,1}$ and $T_{j,2}$ denote the respective first and second factor maps.

In Section~\ref{sec:examples}, we will study examples where the assumptions of Theorem \ref{Thrm: Main result using spectral approach} are verified. Here, we will address first the issue of estimating the quantities $p_k$ given in \eqref{equation: pk}.

We assume that our dynamical system $\bar T:\left[0,1\right]^2 \to \left[0,1\right]^2$ satisfies all the conditions of Theorem \ref{Thrm: Main result using spectral approach}, with $\alpha=1$.
However, we recover the notation used to describe conditions \ref{condition: PE1}-\ref{condition: PE5}. For example, we write that $\bar T_j|_{\overline{U_j}}$ is a bijection between $\overline{U_j}$ and $[0,1]^2$.
The component maps are assumed to be piecewise $C^2$ unidimensional expanding. This implies that each of the factor maps has a $\operatorname{BV}\left(\left[0,1\right]\right)$ invariant probability absolutely continuous with respect to $\mathrm{Leb}$, which we call $h_1$ and $h_2$.
The invariant density of the product map $\bar T$ is  
\begin{equation*}
  h: \left[0,1\right]^2\to \left[0, \infty\right),\,  \left(x,y\right) \mapsto h_1\left(x\right) h_2\left(y\right).
\end{equation*}
The sets $\Delta_n$ can be chosen in the following way:
\begin{equation*}
  \Delta_n=\left\{(x,y) \in [0,1]^2: d\left(x,y\right)<\eta s/n \right\},
\end{equation*}
where
\begin{equation*}
  \eta=\frac{1}{2\int_0^1 h\left(x,x\right) dx}=\frac{1}{2\int_0^1 h_1\left(x\right)h_2\left(x\right) dx}.
\end{equation*}

\subsection{Preparatory results}
Let $D$ be the diagonal of $[0,1]^2$.
We need this simple but useful result:
\begin{lemma} \label{lemma: A_epsilon upper_bound}
  We have the following inclusions:
  \begin{itemize}
    \item $\Delta_n\subseteq \left\{(x,y): d\left(\left(x,y\right), D\right)<\eta s/n \right\}$
    \item $\bar T_j|_{\overline{U_j}}^{-1}(\Delta_n) \subseteq \left\{\left(x,y\right)\in [0,1]^2:d\left(\left(x,y\right), \bar T_j^{-1}D\right)<\eta s/n \right\}$, in case $\eta s/n\leq \varepsilon_0$.
  \end{itemize}
\end{lemma}
\begin{proof}
  To prove the first inclusion, we note that, since $d$ is the Euclidean distance,
  \begin{equation*}
    \left(x,y\right)\in \Delta_n \iff d\left(x,y\right)<\eta s/n \implies d\left(\left(x,x\right), \left(x,y\right)\right)<\eta s/n.
  \end{equation*}
  Regarding the second inclusion, if $\bar T_j\left(x,y\right)\in \Delta_n$ then, by the first inclusion, there exists $(w,w)\in D$ such that $d\left(\bar T_j\left(x,y\right), \left(w,w\right)\right)<\eta s/n$. But then
  \begin{align*}
    d\left(\left(x,y\right), \bar T_j^{-1}\left(w,w\right)\right)=d\left(\bar T_j^{-1}\bar T_j\left(x,y\right), \bar T_j^{-1}\left(w,w\right)\right)\\ \leq \sigma\, d\left(\bar T_j\left(x,y\right), \left(w,w\right)\right)<\sigma\cdot\eta s/n<\eta s/n,
  \end{align*}
  which concludes the proof.
\end{proof}

Under the conditions described above, we can prove the following result:
\begin{proposition}
  If $\bar T_j|_{\overline{U_j}}^{-1}D \cap D=\emptyset$, then there exists $n_0$ such that for all $n\geq n_0$, the intersection $\bar T_j|_{\overline{U_j}}^{-1} \Delta_n \cap \Delta_n$ is empty.
\end{proposition}
\begin{proof}
  We first observe that the sets $\Delta_n$ and $\bar T_j|_{\overline{U_j}}^{-1} \Delta_n$ are shrinking neighborhoods of $D$ and $\bar T_j|_{\overline{U_j}}^{-1}D$, respectively. Since these sets are compact and disjoint, the distance between them is strictly positive, so we can find $n_0$ such that for all $n\geq n_0$, the intersection $\bar T_j|_{\overline{U_j}}^{-1} \Delta_n \cap \Delta_n$ is empty.
\end{proof}

\subsection{Intersecting points with distinct derivatives}

Before going through more results, it will be useful to study the following inequality of real numbers
\begin{equation}
\label{eq:inequality}
  \left| az+b\frac{z^2}{2}\right|\ge \frac{\left|za\right|}{2}.
\end{equation}
If $ab=0$, then every real number satisfies the inequality. On the other hand, if $ab>0$, the set of solutions of the inequality is $\left(-\infty, -3a/b \right] \cup \left[-a/b, \infty\right)$. In the case of $ab<0$, this set is $\left(-\infty, -a/b\right] \cup \left[-3a/b, \infty\right)$.
Hence, the inequality is satisfied for all $z\in \mathbb{R}$, such that $|z|\leq \left|\frac{a}{b} \right|$.
Another important remark is that if we replace $b$ by $b'$ in the inequality, where $|b'|\leq|b|$, then the set $\left\{z\in \mathbb{R}: |z|\leq \left|\frac{a}{b}\right| \right\}$ is still contained in the set of solutions of the inequality.

Recall that the factors of $\bar T_j: V_j\to \mathbb{R}^2$ are denoted by $T_{j,1}$ and $T_{j,2}$. Since the domains of the maps $\bar T_j$ are rectangles, it is not difficult to see what the domains of the maps $T_{j,1}$ and $T_{j,2}$ are intervals.
We now have all the ingredients to prove a lemma, which, in loose terms, states that if there is an intersection of $D$ with $\bar T^{-1}D$ at the point $(x_0,x_0)$, with $T_{j,1}'\left(x_0\right) \neq T_{j,2}'\left(x_0\right)$, then $\bar T^{-1}D$ ``moves away from'' $D$ sufficiently fast.

\begin{lemma} \label{lemma: intersection_diff_derivative}
  Assume $\bar T$ is as described in the beginning of this section. Assume that $(x_0, x_0)\in \overline{U_j}$ is such that $T_{j,1}(x_0)=T_{j,2}(x_0)$, but $T_{j,1}'\left(x_0\right) \neq T_{j,2}'\left(x_0\right)$. Define
  \begin{itemize}
    \item $b_j:=\sup_{\left(x,y\right)\in [0,1]^2} \left|\left[\left(T_{j,1}^{-1}\right)'\left(x\right)\right]^3T_{j,1}''\left(T_{j,1}^{-1}\left(x\right)\right)-\left[\left(T_{j,2}^{-1}\right)'\left(y\right)\right]^3T_{j,2}''\left(T_{j,2}^{-1}\left(y\right)\right) \right|$
    \item $c_j:= \frac{\left|\left(T_{j,1}^{-1}\right)'\left(T_{j,1}(x_0)\right)-\left(T_{j,2}^{-1}\right)'\left(T_{j,2}(x_0)\right) \right|}{b_j}$.
  \end{itemize}
  Then, for all $t$ such that $|t|\leq c_j$ and $(T_{j,1}(x_0)+t, T_{j,2}(x_0)+t)\in \bar T_j(\overline{U_j})$, we have that
  \begin{equation*}
    \left|T_{j,1}^{-1}\left(T_{j,1}(x_0)+t\right)- T_{j,2}^{-1}\left(T_{j,2}(x_0)+t\right)\right| \ge\frac{\left|t\right|}{2}\left|\left(T_{j,1}^{-1}\right)'\left(T_{j,1}(x_0)\right)- \left(T_{j,2}^{-1}\right)'\left(T_{j,2}(x_0)\right) \right|.
  \end{equation*}
\end{lemma}
\begin{proof}
  By Taylor expansion, we have
  \begin{multline*}
    \left|T_{j,1}^{-1}\left(T_{j,1}(x_0)+t\right)-T_{j,2}^{-1}\left(T_{j,2}(x_0)+t\right) \right|=\\
    \Bigg|\left[\left(T_{j,1}^{-1}\right)'\left(T_{j,1}(x_0)\right)-\left(T_{j,2}^{-1}\right)'\left(T_{j,2}(x_0)\right)\right]t\\+ \left[\left(T_{j,1}^{-1}\right)''\left(T_{j,1}(x_0)+r_1\right)-\left(T_{j,2}^{-1}\right)''\left(T_{j,2}(x_0)+r_2\right)\right]\frac{t^2}{2} \Bigg|.
  \end{multline*}
  The result follows by applying the observation above about the inequality \eqref{eq:inequality}, with
  \begin{equation*}
    a=\left(T_{j,1}^{-1}\right)'\left(T_{j,1}(x_0)\right)-\left(T_{j,2}^{-1}\right)'\left(T_{j,2}(x_0)\right)
  \end{equation*}
  and
  \begin{equation*}
    b=\sup_{\left(x,y\right)\in [0,1]^2} \left|\left[\left(T_{j,1}^{-1}\right)'\left(x\right)\right]^3T_{j,1}''\left(T_{j,1}^{-1}\left(x\right)\right)-\left[\left(T_{j,2}^{-1}\right)'\left(y\right)\right]^3T_{j,2}''\left(T_{j,2}^{-1}\left(y\right)\right) \right|.
  \end{equation*}
  The upper bound $b$ is obtained by differentiating twice the expression $f\left(f^{-1}\left(x\right)\right)=x$.
\end{proof}

As a consequence we obtain the following corollary.
\begin{corollary} \label{corollary: intersection_diff_derivative}
 Let 
 $$\textstyle 0<\varepsilon<\frac{\left|\left(T_{j,1}^{-1}\right)'\left(T_{j,1}(x_0)\right)-\left(T_{j,2}^{-1}\right)'\left(T_{j,2}(x_0)\right) \right|^2}{2b_j}\quad\text{and}\quad\delta=\frac{2\varepsilon}{\left|\left(T_{j,1}^{-1}\right)'\left(T_{j,1}(x_0)\right)-\left(T_{j,2}^{-1}\right)'\left(T_{j,2}(x_0)\right) \right|}.$$ Then, for all $t\in \left(\delta, c_j\right) \cup \left(-c_j, -\delta\right)$, we have
  \begin{equation*}
    \left|T_{j,1}^{-1}\left(T_{j,1}(x_0)+t\right)- T_{j,2}^{-1}\left(T_{j,2}(x_0)+t\right)\right|\ge\varepsilon.
  \end{equation*}
\end{corollary}

\subsection{Isolated intersecting points have negligible contribution for the computation of $\theta$}

Now we prove that under the conditions of Lemma \ref{lemma: intersection_diff_derivative}, $(x_0, x_0)$ is an isolated point of $D \cap \bar T^{-1}D$. Moreover, its contribution for the computation of  the extremal index is negligible.
From now on, by $\bar T_j^{-1}$, we mean $\bar T_j^{-1}|_{\overline{U_j}}$.

\begin{proposition} \label{proposition: isolated_points}
  Suppose that $\bar T$ satisfies the properties described in the beginning of this section and that, for some $j$, $(x_0,x_0)\in \overline{U_j} \cap D \cap \bar T^{-1}_j (D)$ is a point for which $T_{j,1}'\left(x_0\right)\neq T_{j,2}'\left(x_0\right)$.
  Then there exists a neighbourhood $V$ of $(x_0, x_0)$ such that
  \begin{equation*}
    V\cap D \cap \bar T_j^{-1} D=\left\{\left(x_0, x_0\right)\right\}
 \quad  \text{and}\quad 
    V\cap D \cap \bar T^{-1} D\subseteq\left\{\left(x_0, x_0\right)\right\}.
  \end{equation*}
  Furthermore,
  \begin{equation*}
    \lim_{n\to\infty} \frac{\mu\left(\Delta_n \cap \bar T^{-1} \Delta_n \cap V\right)}{\mu\left(\Delta_n\right)}=0.
  \end{equation*}
\end{proposition}

\begin{remark}
 The inclusion $V\cap D \cap \bar T^{-1} D\subseteq\left\{\left(x_0, x_0\right)\right\}$ may not be an equality because $(x_0, x_0)$ may belong to $\partial U_j$, where the map can be defined arbitrarily.
\end{remark}

\begin{proof}
  First, suppose that $(x_0, x_0)\in U_j$. We claim that $V=T_j^{-1}\left(B_{c_j}\left(T_{j,1}x_0, T_{j,2} x_0\right)\right)$ is the desired neighborhood. Note that this neighborhood is not necessarily an open set of $[0,1]^2$, but it is on the induced topology of $\overline{U_j}$.
  To prove this claim, we choose $(x,y)\in T_j^{-1}D\cap V\setminus\left\{\left(x_0, x_0\right)\right\}$.
  We know that there exists $t$, where $0<|t|<c_j$, such that $(x,y)=\left(T_{j,1}^{-1}\left(T_{j,1}x_0+t\right),T_{j,2}^{-1}\left(T_{j,2}x_0+t\right)\right)$. Lemma \ref{lemma: intersection_diff_derivative} tells us $(x,y)\not \in D$.
  If $(x_0, x_0)\in\partial U_j$, one should take all $i$ for which $(x_0, x_0)\in D \cap T_j^{-1}D$ and the neighborhood $V_j$ of that point as above. For the $i$ such that $(x_0, x_0) \not \in T_j^{-1} D$, we can take the ball $B_{r_j}$ centered at $(x_0, x_0)$ of radius $r_j:=d\left(\left(x_0, x_0\right), T_j^{-1}D \right)/2$.
  The union of all of these gives a neighborhood of $(x_0, x_0)$ with the desired conditions.
  Of course, we can ensure these neighborhoods are open by taking a ball contained in $V$.
  One important thing to mention is that
  \begin{equation*}
    V\subseteq \bigcup_{j: (x_0,x_0)\in \overline{U}_j} \overline{U_j}.
  \end{equation*}
  For each $j$ such that $\left(x_0, x_0\right)\in \overline{U_j}$, define the following quantities:
  \begin{equation}
    \Gamma_j= \frac{\left|\left(T_{j,1}^{-1}\right)'\left(T_{j,1}x_0\right)-\left(T_{j,2}^{-1}\right)'\left(T_{j,2}x_0\right) \right|^2}{2b_j}
  \end{equation}
  \begin{equation}
    \Lambda_j=\frac{\epsilon_0 \left|\left(T_{j,1}^{-1}\right)'\left(T_{j,1}x_0\right)-\left(T_{j,2}^{-1}\right)'\left(T_{j,2}x_0\right) \right|}{2\sqrt{2}}
  \end{equation}
  we take $\varepsilon_j$ as follows:
  \begin{equation}
    \varepsilon_j<\begin{cases}
      \min\left\{3\epsilon_0, 3r_j \right\} & (x_0,x_0)\not\in T_j^{-1} D \\
      \min\left\{c_j, \Gamma_j, \frac{3d_j}{2}, \Lambda_j \right\}, & (x_0,x_0) \in  T_j^{-1} D
    \end{cases}
  \end{equation}
  where $d_j:=d\left(T_j^{-1} D, D\cap T_j^{-1}\left( \overline{B_{c_j}\left(T_{j,1}x_0, T_{j,2}x_0\right)}\setminus B_{c_j/2}\left(T_{j,1}x_0, T_{j,2}x_0\right)\right) \right)$, which is greater than $0$, since both sets are compact and do not intersect each other.
  Let $\varepsilon=\min_{i: (x_0, x_0)\in \overline{U_j}}\varepsilon_j$.

  Suppose that $(x,y)\in U_j$ where $i$ is such that $(x_0, x_0)\in T_j^{-1}D$.
  Additionally, suppose $(x,y)\in \Delta_n \cap T^{-1}_j \Delta_n\cap T_j^{-1}\left[B_{c_j/2}\left(T_{j,1}x_0, T_{j,2}x_0\right)\right]$ for sufficiently large $n$ (i.e., $1/n \leq \varepsilon/(3\eta s)$).
  We know that $\left(T_{j,1}x, T_{j,2}y\right) \in \Delta_n$, which implies, by Lemma \ref{lemma: A_epsilon upper_bound}, that $d\left(\left(T_{j,1}x, T_{j,2}y\right), D\right)<s\varepsilon/3$.
  Let $(z,z)\in D$ be such that
  \begin{equation*}
    d\left(\left(T_{j,1}x, T_{j,2}y\right), \left(z,z\right)\right)<s\varepsilon/3.
  \end{equation*}
  Then,
  \begin{align*}
    d\left(\left(T_{j,1}x_0, T_{j,2}x_0\right), \left(z,z\right)\right)\leq d\left(\left(T_{j,1}x_0, T_{j,2}x_0\right), \left(T_{j,1}x,T_{j,2}y\right)\right)+d\left(\left(T_{j,1}x, T_{j,2}y\right), \left(z,z\right)\right)\\\leq \frac{c_j}{2}+\frac{s\varepsilon}{3}\leq \frac{c_j}{2}+\frac{\varepsilon}{3}<c_j.
  \end{align*}
  This means that $(z,z)\in B_{c_j}\left(T_{j,1}x_0, T_{j,2}x_0\right)$.
  Let $(u,v)=\bar T_{j}^{-1}\left(z,z\right)$. Since $$(u,v)\in \bar T^{-1}_j B_{c_j}\left(T_{j,1}x_0, T_{j,2}x_0\right),$$
  there exists $|t|<c_j$ such that $$(u,v)=\left(T_{j,1}^{-1}\left(T_{j,1}x_0+t\right), T_{j,2}^{-1}\left(T_{j,2}x_0+t\right)\right).$$
  Note that
  \begin{equation*}
    d\left(u,v\right)\leq d\left(u,x\right)+d\left(x,y\right)+d\left(y,v\right) <s/n \leq \varepsilon.
  \end{equation*}
  Since $\varepsilon<\frac{\left|\left(T_{j,1}^{-1}\right)'\left(T_{j,1}x_0\right)-\left(T_{j,2}^{-1}\right)'\left(T_{j,2}x_0\right) \right|^2}{2b_j}$, Corollary \ref{corollary: intersection_diff_derivative} tells us that
  \begin{equation*}
    |t|<\delta=\frac{2\varepsilon}{\left|\left(T_{j,1}^{-1}\right)'\left(T_{j,1}x_0\right)-\left(T_{j,2}^{-1}\right)'\left(T_{j,2}x_0\right) \right|}.
  \end{equation*}
  Therefore,
  \begin{equation*}
    d\left(\left(z, z\right), \left(T_{j,1}x_0, T_{j,2}x_0\right)\right)< \sqrt{2}\cdot\delta.
  \end{equation*}
  Since by the choice of $\varepsilon$, we have $\sqrt{2}\cdot\delta<\epsilon_0$, then
  \begin{align*}
    d\left(\left(u,v\right), \left(x_0, x_0\right)\right)=d\left(\bar T_{j}^{-1}\left(z, z\right), T_j^{-1}\left(T_{j,1}x_0, T_{j,2}x_0\right)\right) \\\leq \sigma\, d\left(\left(z, z\right), \left(T_{j,1}x_0, T_{j,2}x_0\right)\right)<\sqrt{2}\cdot\delta.
  \end{align*}
  Finally,
  \begin{align*}
    d\left(\left(x,y\right), \left(x_0, x_0\right)\right)\leq d\left(\left(x,y\right), \left(u,v\right)\right)+d\left(\left(u,v\right), \left(x_0, x_0\right)\right)<\\\varepsilon\left(\frac{1}{3}+ \frac{2\sqrt{2}}{\left|\left(T_{j,1}^{-1}\right)'\left(T_{j,1}x_0\right)-\left(T_{j,2}^{-1}\right)'\left(T_{j,2}x_0\right) \right|}\right).
  \end{align*}

  On the other hand, if $(x,y)\in T_j^{-1}\left[B_{c_j}\left(T_{j,1}x_0, T_{j,2}x_0\right)\setminus B_{c_j/2}\left(T_{j,1}x_0, T_{j,2}x_0\right)\right]$, then
  \begin{equation*}
    d\left(\left(x,y\right), D\right)\ge \varepsilon/3 \text{ or } d\left(\left(x,y\right), T_j^{-1}D\right)\ge \varepsilon/3.
  \end{equation*}
  Otherwise, by the triangle inequality, we would get $d_j\leq\frac{2\varepsilon}{3}$, contradicting the choice of $\varepsilon$.
  So, by Lemma \ref{lemma: A_epsilon upper_bound}, we get $(x,y) \not \in \Delta_n \cap \bar T^{-1}_j \Delta_n$.

  In case $(x_0, x_0)\not \in \bar T_j^{-1}D$, we have that
  \begin{equation*}
    \bar T_j^{-1}\Delta_n\subseteq \left\{\left(x,y\right)\in \overline{U_j}:d\left(\left(x,y\right), \bar T_j^{-1}D\right)<\varepsilon/3 \right\}.
  \end{equation*}
  Note that
  \begin{equation*}
    2r_j=d\left(\left(x_0,x_0\right), \bar T_j^{-1}D\right)\leq d\left(\left(x_0,x_0\right), \left(x,y\right)\right)+ d\left(\left(x,y\right), \bar T_j^{-1}D\right),
  \end{equation*}
  which means that if $(x,y)\in V \cap U_j$, then
  \begin{equation*}
    d\left(\left(x,y\right), \bar T_j^{-1}D\right)\ge r_j>\varepsilon/3,
  \end{equation*}
  which implies that $(x,y) \not \in \bar T_j^{-1} \Delta_n$.

  In conclusion,
  \begin{equation*}
    \mu\left(\Delta_n \cap \bar T^{-1}\Delta_n \cap V\right) \leq \sum_{i: (x_0, x_0)\in \overline{U_j}}\mu\left(\Delta_n \cap \bar T_j^{-1}\Delta_n \cap V\right) =\mathcal{O}\left(n^{-2}\right),
  \end{equation*}
  since by the positivity assumption and the fact that $h\in V_\alpha\subseteq L^\infty$,
  \begin{equation*}
    \mu\left(B_{1/n}\left(x_0, x_0\right)\right)=\mathcal{O}\left(n^{-2}\right).
  \end{equation*}
  The conclusion follows by the fact that $\mu\left(\Delta_n\right)=\mathcal{O}\left(n^{-1}\right)$.
\end{proof}

\subsection{Corollaries and sufficient conditions for $\theta=1$}

Now, we point out some corollaries of Proposition \ref{proposition: isolated_points}.
An important thing to mention is that if $\bar T$ satisfies the properties described in the beginning of this section, so do the iterates of $\bar T$. This means that we can also apply Proposition \ref{proposition: isolated_points} to the iterates of $\bar T$.
With this in mind, we have the following:

\begin{corollary} \label{corollary: finite number of T-jD cap D}
  Suppose that for all $\kappa\in \mathbb{N}$, the number of points of $\bigcup_j \bar T_j^{-\kappa}D\cap D$ is finite. If for all these points the derivatives of the components are different, then the extremal index is 1.
\end{corollary}
\begin{proof}
  For each $(x_k, x_k)\in \bigcup_j\bar T_j^{-1}D\cap D$, let $V_k$ be the neighborhood as in Proposition \ref{proposition: isolated_points}, which we assume, without loss of generality, to be an open set of $\left[0,1\right]^2$.
  Let $W=\left(\bigcup_k V_k\right)^c$.
  We note that $\bigcup_j\bar T_j^{-1}D \cap W$ is a compact set whose intersection with $D$ is empty. So we can take $n_0$ large enough such that
  \begin{equation*}
    \bigcup_j\bar T_j^{-1}\Delta_n \cap \Delta_n \cap W=\emptyset \quad \text{for all } n\geq n_0.
  \end{equation*}
  Note that, up to Lebesgue measure zero, and consequently $\mu$-measure zero,
  \begin{equation*}
    \bigcup_j\bar T_j^{-1}\Delta_n= \bar T^{-1} \Delta_n.
  \end{equation*}
  This allows us to make the following estimation
  \begin{equation*}
    \mu\left(\bar T^{-1}\Delta_n\cap \Delta_n \right) \leq \sum_k \mu\left(\bar T^{-1} \Delta_n \cap \Delta_n \cap V_k \right)+\mu\left(\bar T^{-1}\Delta_n \cap \Delta_n \cap W \right).
  \end{equation*}
  Dividing by $\mu\left(\Delta_n\right)$ and using Proposition \ref{proposition: isolated_points}, we conclude that $p_0=0$, where $p_\kappa$ is as defined in \eqref{definition:p_k}.
  For $\kappa\ge 0$, we can apply the same reasoning.
  Since
  \begin{equation*}
    \mu\left(\Delta_n\cap \bar T^{-1}\Delta_n^c \cap\cdots\cap \bar T^{-\kappa}\Delta_n^c \cap \bar T^{-(\kappa+1)}\Delta_n\right)\le \mu\left(\Delta_n\cap \bar T^{-(\kappa+1)}\Delta_n\right),
  \end{equation*}
  it follows that $p_\kappa=0$.
\end{proof}

\begin{corollary} \label{corollary: lower_bound for diference of the derivtives}
  Let $\ell\ge 1$. Assume that there exists $k_\ell>0$ such that for all intersection points $(x_0, x_0)\in D \cap \bigcup _j\bar T_j^{-\ell}D$,
  \begin{equation*}
    \left|\left(T_{j, 1}^{-\ell}\right)'\left( T_{j,1}^\ell x_0 \right)-\left(T_{j,2}^{-\ell}\right)'\left( T_{j,2}^\ell x_0\right)\right|>k_\ell.
  \end{equation*}
  Then, the number of points of $D\cap \bigcup_j\bar T^{-\ell}D$ is finite, which implies that $p_{\ell-1}=0$.
\end{corollary}
\begin{proof}
  It is enough to prove that, for each $i$, $D\cap \bar T_j^{-1}D$ is finite.
  By the proof of the proposition, if $(x_0, x_0)\in D\cap \bar T_j^{-1}D$, then it is the unique point of that set when intersected with $\bar T_j^{-1}\left(B_{c_j}\left(T_{j,1}x_0, T_{j,2} x_0\right)\right)$.
  But $c_j>k_1/b_j$, so $(x_0, x_0)$ is the unique point in
  \begin{equation*}
    D\cap \bar T_j^{-1}D\cap \bar T_j^{-1}\left(B_{k_1/b_j}\left(T_{j,1}x_0, T_{j,2} x_0\right)\right).
  \end{equation*}
  What remains to prove is that, for all $\varepsilon>0$, there is $l_\varepsilon>0$ such that for all $(x,x)\in \overline{U_j} \cap D$, the curve $\bar T_j^{-1} \left[B_{\varepsilon}\left(x,x\right) \cap D\right]$ has a length greater than $l_\varepsilon$.
  This means that if $(x_0, x_0)$ and $(x_1, x_1)$ are different points of $\bar T_{j}^{-1}D$, then $(x_1,x_1)\not \in \bar T_j^{-1}\left(B_{k_1/b_j}\left(T_{j,1}x_0, T_{j,2} x_0\right)\right)$, which implies, by the construction of the proof, that the path in $\bar T_j^{-1}D$ whose end-points are $(x_0, x_0)$ and $(x_1, x_1)$ has length greater than $l_{k_1/b_j}/2$.
  This implies that there is only a finite number of such points.
  To get the lower bound for the length of the curve, we note that
  \begin{equation*}
    l\left(\bar T_j^{-1}\left[B_\varepsilon\left(x, x\right)\cap D\right]\right)= \int_{-\varepsilon}^\varepsilon\sqrt{\left[\left(T_{j,1}^{-1}\right)'\left(x+t\right)\right]^2+\left[\left(T_{j,2}^{-1}\right)'\left(x+t\right)\right]^2}\,dt.
  \end{equation*}
  The domain of integration might not be exactly $(-\varepsilon, \varepsilon)$ as $B_\varepsilon\left(x, x\right)\cap D$ might not be contained in $\overline{U_j}$. But it will always be an interval of length at least $\varepsilon$.
  Suppose that the domain of integration is contained in $\overline{U_j}$. Then, using the fact that $\bar T_j$ is $C^1$ and invertible in an open set containing $\overline{U_j}$, we conclude, by the formula of the derivative of the inverse function, that $\left(T_{j,1}^{-1}\right)'$ and $\left(T_{j,2}^{-1}\right)'$ are both bounded away from zero, allowing us to obtain a lower bound for $l\left(\bar T_j^{-1}\left[B_\varepsilon\left(x, x\right)\cap D\right]\right)$.
\end{proof}

\section{The Extremal Index as a measure of dynamical compatibility}
\label{sec:compatibility-EI}

In this section, we show how the Extremal Index distinguishes between compatible and incompatible factor dynamics, quantifying the degree of compatibility. 

We will consider 2-dimensional linear dynamics, which will simplify the computation of the coefficients $p_k$ given in \eqref{definition:p_k}, which are crucial to compute the Extremal Index. These systems clearly fit the framework of Theorem~\ref{Thrm: Main result using spectral approach}, as they can be seen as particular cases of the class of maps studied in Section~\ref{subsec:piecewise-expanding-example}, for example.

We will start by verifying that if the linear maps have different coefficients then the Extremal Index is 1, which could be obtained easily from the results in Section~\ref{sec:EI}. However, we perform a direct analysis in order to provide further insight about what is behind this interpretation of the role of the Extremal Index and at the same time take the opportunity to set the notation and establish a benchmark against which the cases of full and partial compatibility can be compared.

Full compatibility is obtained by considering the same linear coefficients for the two factor maps, which will give rise to an Extremal Index less than 1, where we recover the formula obtained in previous works, \cite{dq}. Then, we will introduce an example of different linear factor maps, which coincide in some part of phase space, creating a relevant compatibility detected by an Extremal Index less than 1, which will turn out to be a weighted average of the two previous cases.

\subsection{Direct product of different linear maps}
Consider the map
\begin{align}
\label{eq:linear-map}
  T\colon [0,1]^2& \longrightarrow [0,1]^2\\ \nonumber
 (x,y)&\longmapsto \left(\mathbf{a} x \mod 1, \mathbf{b} y \mod 1\right)
\end{align}
where $\mathbf{a}, \mathbf{b} \in \mathbb{Z}\setminus\left\{-1, 0,1\right\}$. The invariant measure $\mu$ is Lebesgue measure.
In this case, the observable function giving rise to the stochastic process, can be written in the simplified form:
 \begin{align}
 \label{eq:observable}
  \varphi\colon [0,1]^2& \longrightarrow \R\\ \nonumber
 (x,y)&\longmapsto -\log d(x,y) 
\end{align}
For each $\tau>0$, define
\begin{equation}
	u_n\left(\tau\right)=\log \frac{2n}{\tau}.
\end{equation}
It is clear that the set $\Delta_n=\left\{\varphi>u_n \right\}=\left\{\left(x,y\right)\in \left[0,1\right]^2:d\left(x,y\right)<\frac{\tau}{2n}\right\}$. We denote the quantity $\frac{\tau}{2n}$ by $r_n$.
It is also clear, since $\mu$ is the Lebesgue measure, that $\mu\left(\Delta_n\right) \sim \frac{\tau}{n}$.
The map $T$ can be partitioned into $\left|\mathbf{a}\mathbf{b}\right|$ open sets, where $T$ restricted to these parts is a diffeomorphism.
Instead of denoting these elements by $U_i$ as in the previous chapter, we will
denote them by $I_{kl}$, where $k \in \{0,...,\mathbf{a}-1\}$ or $k\in\{-1,...,\mathbf{a}\}$ and $l \in \{0,...,\mathbf{b}-1\}$ or $l\in\{-1,..., \mathbf{b}\}$, depending on whether $\mathbf{a}$ and $\mathbf{b}$ are negative or positive.
Similarly, we denote the elements of the analogous partition $T^j$ by $I_{kl}^j$, where $0 \leq k \leq\mathbf{a}^j-1$ or $-1\ge k \ge \mathbf{a}^j$, and $0 \leq l \leq \mathbf{b}^j-1$ or $-1\ge l\ge \mathbf{b}^j$, depending on whether $\mathbf{a}^j$ and $\mathbf{b}^j$ are positive or negative (the sign that matters is the one of $\mathbf{a}^j$, not $\mathbf{a}$).
Explicitly,
\begin{equation}
	I_{kl}^j=\left\{\left(x,y\right)\in \left[0,1\right]^2: x\in I_k^j, \, y\in J_l^j \right\},
\end{equation}
where $I_k^j=\left(\frac{k}{\mathbf{a}^j}, \frac{k+1}{\mathbf{a}^j}\right)$ or $I_k^j=\left(\frac{k+1}{\mathbf{a}^j}, \frac{k}{\mathbf{a}^j}\right)$, depending on whether $\mathbf{a}^j$ is positive or negative. The set $J_l^j$ is defined similarly, but replacing $\mathbf{a}$ by $\mathbf{b}$.
We define the maps
\begin{equation}
	T_{kl}^j:\mathbb{R}^2\to \mathbb{R}^2, \left(x,y\right) \mapsto \left(\mathbf{a}^j x-k, \mathbf{b}^j y-l\right).
\end{equation}
Clearly, $T^j|_{I^j_{kl}}=T^j_{kl}|_{I^j_{kl}}$.
We denote the diagonal of $\left[0,1\right]^2$ by $D$.

In this section, we verify that when $\left|\mathbf{a}\right|\neq \left| \mathbf{b}\right|$, then Extremal Index is $\theta=1$.
This can be explained by the fact that, for a fixed $j\in\mathbb{N}$, there is only a finite number of points on the diagonal that return to the diagonal after $j$ iterations. Since the derivatives of the component maps are different, then the area of a small neighbourhood of those points that return to a neighbourhood of the diagonal is negligible for the computation of the Extremal Index.

The proof will be based on checking that the following condition holds:
\begin{equation}
\label{eq:D'}
\lim_{n\to\infty}n\mu\left(\Delta_n\cap T^{-j}(\Delta_n)\right)=0, \quad\text{for all}\quad j\in\N.
\end{equation}
Note that validity of \eqref{eq:D'} implies that $p_k=0$, for all $k\in\N_0$ and therefore $\theta=1-\sum_{k\geq0} p_k=1$.

We start by noting that $T^{-j}\left(D\right) \cap I_{kl}^j$ is the line
\begin{equation*}
	\left( y- \frac{l}{\mathbf{b}^j}\right)=\frac{\mathbf{a}^j}{\mathbf{b}^j}\left(x-\frac{k}{\mathbf{a}^j}\right)
\end{equation*}
intersected with $I_{kl}^j$.
The set $T^{-j}\left(\Delta_n\right) \cap I_{kl}^j$ is the space between the lines
\begin{equation*}
	y-\frac{l+\left(\frac{t}{2n}\right)}{\mathbf{b}^j}=\frac{\mathbf{a}^j}{\mathbf{b}^j}\left(x-\frac{k}{\mathbf{a}^j}\right);
\qquad	y-\frac{l}{\mathbf{b}^j}=\frac{\mathbf{a}^j}{\mathbf{b}^j}\left(x-\frac{k+\left(\frac{t}{2n}\right)}{\mathbf{a}^j}\right).
\end{equation*}
So, $T^{-j}\left(\Delta_n\right) \cap I_{kl}^j$ is the set of points $\left(x,y\right)\in I_{kl}^j$ such that
\begin{equation*}
	\left|\left( y- \frac{l}{\mathbf{b}^j}\right)-\frac{\mathbf{a}^j}{\mathbf{b}^j}\left(x-\frac{k}{\mathbf{a}^j}\right)\right| < \frac{r_n}{\left|\mathbf{b}\right|^j}.
\end{equation*}
We note that if $D \cap T^{-j}\left(D  \right) \cap I_{kl}^j \neq \emptyset$, then it contains only one point, since it corresponds to the intersection of two linear maps with different derivatives.
To prove condition \eqref{eq:D'}, we need an upper bound for the measure of the sets $\Delta_n\cap T^{-j}\Delta_n$. To get it, we analyse these sets restricted to $I^j_{kl}$.
\begin{lemma} \label{lemma: upper bound for An T-jAn}
	Suppose that $T:[0,1]^2\to[0,1]^2$ is a map and that there exists an open set $A$ such that $T^j|_A$ is a diffeomorphism between $A$ and $\operatorname{int}\left([0,1]^2\right)$ of the form $T^j|_A(x,y)=(\mathbf{a}^j x-q, \mathbf{b}^j y-p)$, with $\left|\mathbf{a}\right| \neq \left|\mathbf{b}\right|$. Suppose also that $D \cap \overline{T^{-j}\left(D\right) \cap A}=(x_0,x_0)$. Then, the set $\Delta_n \cap T^{-j}\left(\Delta_n\right) \cap A$ is a subset of
	\begin{equation*}
		\left\{(x,y) \in A: d(x,x_0) \leq \frac{r_n\left(1+\left|\mathbf{b}\right|^j\right)}{\left|\left|\mathbf{b}\right|^j-\left|\mathbf{a}\right|^j\right|}, \,d(y,x_0) \leq \frac{r_n \left(1+\left|\mathbf{a}\right|^j\right)}{\left|\left|\mathbf{b}\right|^j-\left|\mathbf{a}\right|^j\right|}\right\},
	\end{equation*}
	which is non-empty.
\end{lemma}
\begin{proof}
	Suppose that $(x,y) \in \Delta_n \cap T^{-j}\left(\Delta_n\right) \cap A$ and that $\left(x_0, x_0\right)\in \overline{T^{-j}D\cap A}$.
	We note that $T^j|_A$ can be extended to $\mathbb{R}^2$, so we will assume that $T^j\left(x_0, x_0\right)=\left(T_1^j x_0, T_2^j x_0 \right)$ and $T^j_1x_0=T_2^jx_0$, where
	\begin{equation*}
		T_1^j(x)=\mathbf{a}^jx-q
	\qquad
		T_2^j(x)=\mathbf{b}^j x-p.
	\end{equation*}
	We want to estimate $d(x,x_0)$ and $d(y,x_0)$. Assume first that $\left|\mathbf{b}\right|^j>\left|\mathbf{a}\right|^j$. Then,
	\begin{align*}
		d(y,x_0)\leq&\frac{d\left(T_2^jy,T_2^j x_0\right)}{\left|\mathbf{b}\right|^j} \leq \frac{d\left(T_1^jx,T_2^j y\right)+d\left(T_1^jx,T_1^j x_0\right)}{\left|\mathbf{b}\right|^j}
		\leq \frac{r_n+ \left|\mathbf{a}\right|^jd\left(x, x_0\right)}{\left|\mathbf{b}\right|^j}\\ 
		\leq&\frac{r_n+ \left|\mathbf{a}\right|^j\left(d\left(y, x\right)+d\left(y, x_0\right)\right)}{\left|\mathbf{b}\right|^j} 
	\end{align*}
	Therefore,
	\begin{equation*}
		d(y,x_0) \leq \frac{r_n \left(1+\left|\mathbf{a}\right|^j\right)}{\left|\mathbf{b}\right|^j-\left|\mathbf{a}\right|^j}.
	\end{equation*}
	By a simple triangle inequality, we conclude that
	\begin{equation*}
		d(x,x_0) \leq d(y,x)+d(y,x_0)\leq \frac{r_n\left(1+\left|\mathbf{b}\right|^j\right)}{\left|\mathbf{b}\right|^j-\left|\mathbf{a}\right|^j}.
	\end{equation*}
	The proof of the case where $\left|\mathbf{a}\right|^j>\left|\mathbf{b}\right|^j$ is similar.
\end{proof}

Define the quantity
\begin{multline}
\label{eq:R'}
	\beta_n=\max\left\{m \in \mathbb{N}: \forall 1 \leq j \leq m: T^{-j}\Delta_n \cap \Delta_n \cap I_{kl}^j \neq \emptyset \iff D \cap \overline{T^{-j} D \cap I_{kl}^j} \neq \emptyset,\right.\\ \left.\forall k,l \text{ where $I_{kl}^j$ is defined.} \right\} 
\end{multline}
\begin{lemma} \label{lemma: R'n to infty}
	$\displaystyle \lim_{n \to \infty} \beta_n= +\infty$.
\end{lemma}
\begin{proof}
	Note that $\beta_n$ is an increasing sequence of integers. So, the only way its limit is not infinity is if the sequence is constant after a certain $n$. Assume that this is indeed the case.
	This means that there exists $j,k,l \in \mathbb{N}$ such that, for all $n \in \mathbb{N}$, $T^{-j}\Delta_n \cap \Delta_n \cap I_{kl}^j \neq \emptyset$, but $D \cap \overline{T^{-j}D \cap I_{kl}^j}=\emptyset$.
	Both $D$ and $\overline{T^{-j}D \cap I_{kl}^j}$ are compact, which means that if the intersection of these sets is empty, then $d\left(D, \overline{T^{-j}D \cap I_{kl}^j}\right) >0$.
	Take $n$ such $\frac{\tau}{n}<d\left(D, \overline{T^{-j}D \cap I_{kl}^j}\right)$.
	Then, if $(x,y) \in T^{-j}\Delta_n \cap \Delta_n \cap I_{kl}^j$, we have that
	\begin{equation*}
		d\left(D, \overline{T^{-j}D \cap I_{kl}^j}\right) \leq d\left(D, \left(x,y\right)\right)+ d\left(\overline{T^{-j}D \cap I_{kl}^j}, \left(x,y\right)\right) < 2\frac{\tau}{2n}=\frac{\tau}{n},
	\end{equation*}
	which yields a contradiction.
	To justify the last inequality, note that:
	\begin{equation*}
		\left(x,y\right)\in \Delta_n \implies |x-y|<\frac{\tau}{2n} \implies d\left(\left(x,y\right), D\right)\leq\frac{\tau}{2n}=r_n,
	\end{equation*}
	and also that
	\begin{equation*}
		\left(x,y\right) \in T^{-j}\Delta_n \cap I_{kl}^j \implies \left|\left( y- \frac{l}{\mathbf{b}^j}\right)-\frac{\mathbf{a}^j}{\mathbf{b}^j}\left(x-\frac{k}{\mathbf{a}^j}\right)\right| < \frac{r_n}{\left|\mathbf{b}\right|^j}.
	\end{equation*}
	Since $\left(x,y\right) \in I_{kl}^j$, then $\left(x, \frac{\mathbf{a}^j}{\mathbf{b}^j}\left(x-\frac{k}{\mathbf{a}^j}\right)+\frac{l}{\mathbf{b}^j}\right) \in I_{kl}^j \cap T^{-j}\left(D\right)$.
	This means that
	\begin{equation*}
		d\left(T^{-j}D \cap I_{kl}^j, \left(x,y\right)\right)\leq\frac{r_n}{\left|\mathbf{b}\right|^j}<r_n.
	\end{equation*}
	
	The fact that $d\left(T^{-j}D \cap I_{kl}^j, \left(x,y\right)\right)=d\left(\overline{T^{-j}D \cap I_{kl}^j}, \left(x,y\right)\right)$ concludes the proof.
\end{proof}

We are now ready to establish the following.
\begin{proposition}
	Condition \eqref{eq:D'} is satisfied.
\end{proposition}
\begin{proof}
	The number of points of $\bigcup_{kl}D\cap \overline{T^{-j}D \cap I^j_{kl}}$ is bounded by $\left|\mathbf{a}^j-\mathbf{b}^j \right|+1$, since $(x,x)\in D\cap \overline{T^{-j}D \cap I^j_{kl}} \implies \mathbf{a}^jx-\mathbf{b}^jx \in \mathbb{Z}$.
Since the sets $I_{kl}^j$ are disjoint and their union has full measure, it follows that
\begin{align*}
		&n \mu \left(\Delta_n\cap T^{-j}\Delta_n\right)=
		n \sum_{k,l} \mu \left(\Delta_n\cap T^{-j}\Delta_n \cap I^j_{kl}\right).
	\end{align*} 
	Noting that each point of $\left[0,1\right]^2$ is, at most, in four of the $\overline{ I^j_{kl}}$, and assuming that $n$ is sufficiently large so that $\beta_n>j$ (which we can assume by Lemma~\ref{lemma: R'n to infty}), then we can use Lemma~\ref{lemma: upper bound for An T-jAn} to obtain
	\begin{equation*} \label{equation: intermediate D_0}
		n \mu \left(\Delta_n\cap T^{-j}\Delta_n\right)\leq 4n\left(\left(\left|\mathbf{a}^j-\mathbf{b}^j\right|+1\right)\left(\frac{\left(1+\left|\mathbf{a}\right|^j\right)\left(1+\left|\mathbf{b}\right|^j\right)}{\left(\left|\mathbf{a}\right|^j-\left|\mathbf{b}\right|^j\right)^2}\right)r_n^2\right).
	\end{equation*}
\end{proof}

\subsection{Direct product of identical linear maps}
In this section we consider $T$ to be given as in \eqref{eq:linear-map}, but this time we assume that $\mathbf{b}=\mathbf{a}$. Let the observable $\varphi$ be as before, \eqref{eq:observable}.

We define the quantity $\beta_n$ as in \eqref{eq:R'}. It goes to infinity as $n$ goes to infinity, by Lemma \ref{lemma: R'n to infty}. It is easy to see that condition \eqref{eq:D'} does not hold anymore, but we introduce a new version which implies that $p_k=0$ for all $k\in\N$ and then we will only be left with computing $p_0$ in order to obtain the Extremal Index.

Let $\Delta_n^{(1)}=\Delta_n\cap T^{-1}(\Delta_n^c)$. We will show that
\begin{equation}
\label{eq:D'1}
\lim_{n\to\infty}n\mu\left(\Delta_n^{(1)}\cap T^{-j}(\Delta_n)\right)=0, \quad\text{for all}\quad j\geq 2.
\end{equation}

\begin{proposition} \label{Proposition: D'1 for equal maps}
	Condition \eqref{eq:D'1} is satisfied.
\end{proposition}
\begin{proof}
	The first thing to notice is that
	\begin{equation*}
		\Delta_n^{(1)}=\left\{(x,y) \in [0,1]^2 :\frac{r_n}{\left|\mathbf{a}\right|}<|x-y|<r_n\right\}\cup B,
	\end{equation*}
	where $m\left(B\right)=\mathcal{O}\left(\frac{1}{n^2}\right)$.
	Note that for $j<\beta_n$ the set $I_{kl}^j \cap T^{-j} (\Delta_n) \cap \Delta_n^{(1)}$ is, in fact, empty.
	To see this, observe that
	\begin{equation*}
		(x,y) \in T^{-j}(\Delta_n) \cap I_{kl}^j \implies \left|\left(y- \frac{l}{\mathbf{a}^j}\right)-\left(x-\frac{k}{\mathbf{a}^j}\right)\right| < \frac{r_n}{\left|\mathbf{a}\right|^j}.
	\end{equation*}
	By the choice of $\beta_n$, the intersection $I_{kl}^j \cap T^{-j} (\Delta_n) \cap \Delta_n^{(1)}$ can only be non-empty if $k=l$. But in that case,
	\begin{equation*}
		(x,y)\in \Delta_n^{(1)}\cap T^{-j}(\Delta_n) \cap I_{kk}^j \implies \frac{r_n}{\left|\mathbf{a}\right|}<|x-y|<\frac{r_n}{\left|\mathbf{a}\right|^{j}}, 
	\end{equation*}
	and since $\left|\mathbf{a}\right|>1$, $\Delta_n^{(1)}\cap T^{-j}(\Delta_n) \cap I_{kk}^j=\emptyset$.
	Hence, the conclusion follows trivially.
\end{proof}
We are now able to compute the Extremal Index.
\begin{proposition} \label{proposition: Extremal index equal maps}
	The extremal index is $$\theta=\frac{\left|\mathbf{a}\right|-1}{\left|\mathbf{a}\right|}.$$
\end{proposition}
\begin{proof}
Start by noting that validity of \eqref{eq:D'1} implies that $p_k=0$ for all $k\in\N$. So we are left to compute
	\begin{equation*}
		p_0=\lim_{n \to \infty} 1- \frac{\mu\left(\Delta_n^{(1)}\right)}{\mu\left(\Delta_n\right)}=1-\lim_{n \to \infty} \frac{\frac{(\left|\mathbf{a}\right|-1)}{\left|\mathbf{a}\right|}r_n+\mathcal{O}\left(\frac{1}{n^2}\right)}{r_n}=1-\frac{\left|\mathbf{a}\right|-1}{\left|\mathbf{a}\right|}.
	\end{equation*}
\end{proof}

\subsection{Direct product of partially matching linear maps}

We consider now linear factor maps that coincide only in a given part of the phase space. Namely, we consider
\begin{align*}
  T\colon [0,1]^2& \longrightarrow [0,1]^2\\ \nonumber
 (x,y)&\longmapsto \left(T_1(x), T_2(y)\right)
\end{align*}
where 
\begin{equation*}
	T_1(x)=\mathbf{a} x\mod1 \quad\text{and}\quad T_2(y)= \begin{cases}
		\mathbf{a} y\mod1, & y \leq 1/\left|\mathbf{a}\right| \\
		\mathbf{a}\mathbf{b} y \mod 1, & y \ge 1/\left|\mathbf{a}\right|.
	\end{cases}
\end{equation*}
The invariant measure $\mu$ is still the Lebesgue measure and Theorem \ref{Thrm: Main result using spectral approach} is still valid.
For convenience, we label the partition associated with $T$ differently.
Let $\left\{J_k\right\}_{k=0}^{\left(\left|\mathbf{a}\right|-1\right)\left|\mathbf{b}\right|}$ be the partition associated with $T_2$.
Denote by $C_{i_0,...,i_{j-1}}$ the $j$-cylinders associated with $T_2$.
Note that $y\in C_{i_0,...,i_{j-1}}$ means that $y\in J_{i_0}, \, T_2(y) \in J_{i_1}, \, ..., \, T_2^{j-1}(y) \in J_{i_{j-1}}$.
We denote the elements of the partition associated with $T_1$ by $\left\{I_k^j\right\}_{k=0}^{\left|\mathbf{a}\right|^j-1}$.
As in the previous case, we will prove that condition \eqref{eq:D'1} holds.
\begin{proposition} \label{proposition: EI of maps only equal in part of the domain}
	Condition \eqref{eq:D'1} is satisfied and the Extremal Index is $$\theta=\frac{\left|\mathbf{a}\right|-1}{\left|\mathbf{a}\right|^2}+\frac{\left|\mathbf{a}\right|-1}{\left|\mathbf{a}\right|}.$$
\end{proposition}
\begin{proof}
	Let $\beta_n$ be defined as in \eqref{eq:R'}.
	The map restricted to the open set $I_{0}^j \times C_{00...0}$ is equal to the one studied previously in this section.
	So, if $j<\beta_n$, it was shown in the proof of Proposition \ref{Proposition: D'1 for equal maps} that
	\begin{equation} \label{equation: empty in the zone where the maps are equal}
		\Delta_n^{(1)}\cap T^{-j}\Delta_n^{(1)} \cap I_{0}^j \times C_{00...0}= \emptyset.
	\end{equation}
	In the remaining regions, the maps are different.
	In fact, in those regions, the map $T_2$ assumes the form $\mathbf{a}^j \mathbf{b}^my \mod 1$, for some $m \in \{1,...,j\}$.
	To get an upper bound for the number of points of the diagonal which are mapped into the diagonal, we sum
	\begin{equation*}
		\sum_{m=1}^j \left|\mathbf{a}\right|^j\left(\left|\mathbf{b}\right|^m-1\right)=\frac{\left|\mathbf{b}\right|}{\left|\mathbf{b}\right|-1}\left|\mathbf{a}\right|^j\left(\left|\mathbf{b}\right|^{j}-1\right)-j\left|\mathbf{a}\right|^j,
	\end{equation*}
	where each term of the sum represents the number of points of the diagonal mapped into the diagonal, when we consider the map $(\mathbf{a}^jx \mod1, \mathbf{b}^m \mathbf{a}^j y \mod1)$, after removing the point $(0,0)$, since it is always in the region where both maps are equal.
	
	Using Lemma~\ref{lemma: R'n to infty}, we may assume that $n$ is sufficiently large so that $j\leq \beta_n$ and then, by  Lemma \ref{lemma: upper bound for An T-jAn}, we have
	\begin{equation*} \label{equation: upper bound for An T-jAn zone different}
		\mu\left(\Delta_n^{(1)}\cap T^{-j}\Delta_n\right) \leq \left[\frac{\left|\mathbf{b}\right|}{\left|\mathbf{b}\right|-1}\left|\mathbf{a}\right|^j\left(\left|\mathbf{b}\right|^{j}-1\right)-j\left|\mathbf{a}\right|^j\right] \cdot  \left(\frac{r_n^2\left(1+\left|\mathbf{a} \mathbf{b}\right|^j \right)\cdot\left(1+\left|\mathbf{a}\right|^j\right)}{\left|\left|\mathbf{b}\right|^j-\left|\mathbf{a}\right|^j \right|^2} \right).
	\end{equation*}
	
It follows that 
\begin{equation*}
n\mu\left(\Delta_n^{(1)}\cap T^{-j}\Delta_n\right) \leq nr_n^2 \frac{\left|\mathbf{a} \mathbf{b}\right|^{2j} \cdot \left|\mathbf{a}\right|^j}{\left|\left|\mathbf{b}\right|^j-\left|\mathbf{a}\right|^j \right|^2} \leq nr_n^2\left(\left|\mathbf{a}\right|^3 \left|\mathbf{b}\right|^2 \right)^j\xrightarrow[n\to\infty]{}0
\end{equation*}

As observed earlier, the fact that \eqref{eq:D'1} holds implies that $p_k=0$ for all $k\in\N$, which means that we are again left with computing $p_0$.	For that matter,  note that
	\begin{equation*}
		\mu\left(\Delta_n^{(1)}\right)=\mu\left(\Delta_n^{(1)}\cap I_{0} \times C_{0}\right)+\mu\left(\Delta_n^{(1)}\cap \left(I_{0} \times C_{0}\right)^c\right).
	\end{equation*}
	The first term is computed as in Proposition \ref{proposition: Extremal index equal maps}, except it comes with an extra factor of $\frac{1}{\left|\mathbf{a} \right|}$, which is precisely the measure of the set $\left\{x\in \left[0,1\right]: T_1(x)=T_2(x) \right\}$. So, the first term is $\frac{r_n\left(\left|\mathbf{a}\right|-1\right)}{\left|\mathbf{a}\right|^2}+\mathcal{O}\left(\frac{1}{n^2}\right)$. The second term is $\frac{r_n\left(\left|\mathbf{a}\right|-1\right)}{\left|\mathbf{a}\right|}-\mathcal{O}\left(\frac{1}{n^2}\right)$. This is because the measure of $\Delta_n \cap \left(I_0\times C_0\right)^c$ is $r_n\left( \frac{\left|\mathbf{a}\right|-1}{\left|\mathbf{a}\right|}\right)+\mathcal{O}\left(\frac{1}{n^2}\right)$ and, by Lemma \ref{lemma: upper bound for An T-jAn} and by the fact that only a finite number of points of $D\cap \left(I_0\times C_0\right)^c$ return to $D$ after the first iteration, which explains that there is only a $\mathcal{O}\left(\frac{1}{n^2}\right)$ contribution.
	This means that the extremal index is $\theta=\frac{\left|\mathbf{a}\right|-1}{\left|\mathbf{a}\right|^2}+\frac{\left|\mathbf{a}\right|-1}{\left|\mathbf{a}\right|}$.
\end{proof}

The proof of Proposition \ref{proposition: EI of maps only equal in part of the domain} provides clear insight into the mechanisms determining the value of the Extremal Index in this class of dynamics. The first contribution arises from the region where the factor maps coincide. This region represents a relative weight of $\frac{1}{|\mathbf{a}|}$ of the diagonal, and the corresponding fraction of points in $\Delta_n$ that do not return to $\Delta_n$ in one iteration converges to $\frac{|\mathbf{a}| - 1}{|\mathbf{a}|}$. Their combined contribution to the Extremal Index is given by the product of these two quantities.

The complementary region, where the factor maps are misaligned, has relative weight $\frac{|\mathbf{a}| - 1}{|\mathbf{a}|}$ of the diagonal. In this case, since the factor maps differ, the fraction of points in $\Delta_n$ that do not return to $\Delta_n$ converges to $1$.

Equivalently, the Extremal Index can be written as $1 - \frac{1}{|\mathbf{a}|^2}$, indicating that, in the limit, a proportion $\frac{1}{|\mathbf{a}|^2}$ of points in a neighbourhood of the diagonal return to that neighbourhood after one iteration.

\section{Applications}
\label{sec:examples}

In this section, we provide some examples to illustrate the potential of the results obtained earlier, namely Theorem~\ref{Thrm: Main result using spectral approach} and Corollaries~\ref{corollary: finite number of T-jD cap D} and \ref{corollary: lower_bound for diference of the derivtives}. We will introduce a class of perturbations of linear maps that fits this framework and show that when taking the direct product of perturbed versions of two dynamically incompatible maps, in the sense that the set of points that return to the diagonal is negligible, we obtain an Extremal Index equal to 1.

\subsection{Perturbed linear maps}
\label{subsec:perturbed-linear}

We start by recalling the definition of $\alpha$-H\"older function.
Let $f: I\subseteq \mathbb{R}^n\to \mathbb{R}^n$. We say that $f$ is $\alpha$-H\"older if there exists $C>0$ such that, for all $x, y\in I$
\begin{equation}
  \left| f\left(x\right)-f\left(y\right)\right| \leq C \left|x-y\right|^\alpha.
\end{equation}
It is easy to see that if $f$ is $\alpha$-H\"older and bounded away from zero (in norm), then there is $C>0$ such that given $z\in I$, $\epsilon>0$ and $x,y\in B_\epsilon\left(z\right)$,
\begin{equation} \label{equation: alpha holder plus}
  \left| f\left(x\right)-f\left(y\right)\right|\leq C \left| f\left(z\right)\right|\epsilon^\alpha.
\end{equation}
Let $f: I\subseteq \mathbb{R}^n \to \mathbb{R}^n$. We say that $f\in C^{1+\alpha}\left(I\right)$ or simply that $f$ is $C^{1+\alpha}$ if its determinant is $\alpha$-H\"older.

It is clear that the product of expanding linear maps satisfies all the conditions of Theorem \ref{Thrm: Main result using spectral approach}.
Using this approach, we can guarantee that Theorem \ref{Thrm: Main result using spectral approach} can also be applied to small perturbations of expanding linear maps.
To precisely describe what we mean by small perturbations, consider integers $\mathbf{a}, \mathbf{b}$, greater than $1$ in absolute value. Let $\delta<\frac{\min\{\left|\mathbf{a}\right|, \left|\mathbf{b}\right|\}-1}{2}$.
We label the differentiable regions (open sets) of the map
\begin{align*}
  T\colon [0,1]^2& \longrightarrow [0,1]^2\\
 (x,y)&\longmapsto \left(\mathbf{a} x \mod 1, \mathbf{b} y \mod 1\right)
\end{align*}
by $I_{kl}=I_k \times J_l$, where $k \in \{0,\ldots,|\mathbf{a}|-1\}$, and $l\in \{0,\ldots,|\mathbf{b}|-1\}$.
Consider the maps $f_k:I_k' \to \mathbb{R}$ and $g_l:J_l'\to \mathbb{R}$, where $I_k'$ and $J_l'$ are compact sets such that $\overline{I_k}\subseteq\operatorname{int}I_k'$ and $\overline{J_l}\subseteq\operatorname{int}J_l'$. Suppose that $f_k$ and $g_l$ are $C^{1+\alpha}$, for some $\alpha>0$.
Furthermore, assume that $f_k\left(\overline{I_k}\right)\subseteq\left[-1,1\right]$, $g_l\left(\overline{J_l}\right)\subseteq\left[-1,1\right]$, and that $f_k|_{\partial I_k}=0$ and $g_l|_{\partial J_l}=0$. Suppose also that
\begin{equation*}
  \forall x \in I_k', y\in J_l', \left|f_k'(x)\right|,\left|g_l'(y)\right| \leq 1.
\end{equation*}
Then, we consider the dynamical system
\begin{equation}
\label{eq:def-perturbed-linear-map}
  \tilde{T}:[0,1]^2\to [0,1]^2, \tilde{T}|_{I_{kl}}=T|_{I_{kl}}+\delta \left(f_k|_{I_k} \times g_l|_{J_l} \right),
\end{equation}
where 
\begin{align*}
  f_k \times g_l\colon I_k'\times J_l'& \mathbb{R}^2\\
  \left(x,y\right)&\longmapsto \left(f_k (x), g_l (y)\right).
\end{align*}
From now on, in order to simplify notation, we will write $f_k$ instead of $f_k|_{I_k}$ and analogously for $g_l$.
We denote $\tilde{T}_{kl}$ to the extension of $\tilde{T}|_{I_{kl}}$ to $\operatorname{int}\left(I'_k \times J_l'\right)$.
In order to apply Theorem \ref{Thrm: Main result using spectral approach}, we check that $\tilde{T}$ satisfies \ref{condition: PE1}-\ref{condition: PE5}, and that it satisfies Proposition \ref{proposition: density_bounded_from_zero}.

First, we check that each $\tilde{T}_{kl}$ is uniformly expanding and its restriction to $\overline{I_{kl}}$ is a diffeomorphism between $\overline{I_{kl}}$ and $\left[0,1\right]^2$. To do so, it is enough to check that each of the components are uniformly expanding and that they map, respectively, $\overline{I_k}$ and $\overline{J_l}$ to $\left[0,1\right]$. We do it for the first component. The derivative of it is $\mathbf{a}+\delta f_k'(x)$. If $\mathbf{a}>0$, we get $\mathbf{a}+\delta f_k'(x)>\mathbf{a}/2+1/2>1$, so the map is increasing and expanding. In particular, it is injective.
By the boundary conditions imposed on the map $f_k$, we get that its image is $\left[0,1\right]$. If $\mathbf{a}<0$, the proof is similar.

The map $\tilde{T}$ satisfies condition \ref{condition: PE1} and \ref{condition: PE3}, since it was actually constructed on the sets $U_i$ (that here we label by $I_{kl}$) which are the same that are used in the linear setting.
Condition \ref{condition: PE4} follows from the fact that both components are uniformly expanding.
To check condition \ref{condition: PE2}, we need the following propositions:

It remains to check $\tilde{T}$ satisfies Proposition \ref{proposition: density_bounded_from_zero}.
For that purpose, it suffices to show that for each component map of $\tilde{T}$, and for each interval $I\subseteq \left[0,1\right]$, there exists $n\in \mathbb{N}$ such that the image of $I$ under that component is $\left[0,1\right]$ (up to Lebesgue measure zero). By construction of $\tilde{T}$, there exists $\lambda\ge1$ such that the derivative of each component map has absolute value greater than $\lambda$ on each partition element.
An induction argument shows that the image of any interval is either $[0,1]$, an interval of length at least $\lambda$ times the original, or a union of two intervals near $0$ and $1$ whose total length grows. In conclusion, Proposition \ref{proposition: density_bounded_from_zero} applies. Hence, Theorem~\ref{Thrm: Main result using spectral approach} also applies.

Assuming that $\left|\mathbf{a}\right| \neq \left|\mathbf{b} \right|$ and that the maps $f_k$ and $g_l$ are $C^2$, one can easily compute the extremal index. Let $\delta<1/2$. Then, it is clear that for all $x\in \overline{I_k^j}$ and $y\in \overline{J_l^j}$,
\begin{equation*}
  \left(\left|\mathbf{a}\right|-\delta\right)^j\leq\left| \left(T_{i,1}^j\right)'\left(x\right)\right|\leq \left(\left|\mathbf{a}\right|+\delta\right)^j
\end{equation*}
\begin{equation*}
  \left(\left|\mathbf{b}\right|-\delta\right)^j\leq\left| \left(T_{i,2}^j\right)'\left(y\right)\right|\leq \left(\left|\mathbf{b}\right|+\delta\right)^j.
\end{equation*}
This, together with the fact that $\left|\mathbf{a}\right|\neq \left|\mathbf{b}\right|$, implies that the conditions of Corollary \ref{corollary: lower_bound for diference of the derivtives} are satisfied, meaning that the EVL has an extremal index equal to $1$. In other words, we have just proved the following:
\begin{theorem}\label{thm:perturbed-linear-maps}
Let $\tilde T\colon [0,1]^2\to[0,1]^2$ be a perturbed linear map as defined in \eqref{eq:def-perturbed-linear-map}, with $\left|\mathbf{a}\right| \neq \left|\mathbf{b} \right|$. Then, for all $s>0$ there exists a sequence $(u_n)_{n\in\N}$ such that \eqref{equation: un in thrm spec approach} holds and 
  \begin{equation*} 
    \lim_{n\to\infty}\mu\left(M_n\leq u_n\right)=e^{-s},
  \end{equation*}

\end{theorem}

\appendix

\section{The co-dimension $C_q$}
\label{sec:codimension}

Throughout the paper, we mostly considered expanding maps with smooth
invariant densities bounded away from zero and infinity. In that case,
the co-dimension satisfies
\[
C_q=1,
\]
reflecting the fact that the measures behave locally like Lebesgue measure.

However, $C_q$ may differ from $1$ when the invariant densities exhibit
singularities or zeros, or when the measures have a more irregular (e.g.
fractal) structure. This occurs, for instance, in unimodal maps
\cite{rossler,topsynchro}.

We now give an explicit formula for $C_q$ in the case of absolutely continuous
measures whose densities have finitely many singularities or zeros.

\begin{proposition}
\label{codimension}
Let $\mu_1,\dots,\mu_q$ be probability measures on $\mathbb{R}$, absolutely
continuous with respect to Lebesgue measure, with densities $h_1,\dots,h_q$.
Assume that the common support
\[
I=\bigcap_{i=1}^q \mathrm{supp}(\mu_i)
\]
is a finite union of closed intervals.

Suppose that each $h_i$ has at most finitely many singularities or zeros,
located at points $s\in S\subset I$, and that near each $s$,
\[
h_i(x)\asymp |x-s|^{\alpha_i(s)},
\]
with $-1<\alpha_i(s)<0$ for singularities and $\alpha_i(s)>0$ for zeros.

Define
\[
\overline{\alpha}
=
\min_{s\in S}\left\{
\frac1q\sum_{i=1}^q \alpha_i(s)
\right\}.
\]
Then:
\begin{enumerate}
\item If $S=\emptyset$ or $\overline{\alpha}\ge0$, then
\[
C_q(\mu_1,\dots,\mu_q)=1.
\]

\item If $\overline{\alpha}<0$, then
\[
C_q(\mu_1,\dots,\mu_q)
=
\begin{cases}
1 & \text{if } q<-\dfrac{1}{\overline{\alpha}}, \\[6pt]
\dfrac{q(1+\overline{\alpha})}{q-1} & \text{otherwise}.
\end{cases}
\]
\end{enumerate}
\end{proposition}

\begin{proof}
Let
\[
I(r)=\int_I \prod_{i=2}^q \mu_i(B_r(x))\,d\mu_1(x).
\]
We split the integral into contributions near singularities and away from them:
\[
I(r)=I_1(r)+I_2(r),
\]
where
\[
I_1(r)=\sum_{s\in S}\int_{B_r(s)} h_1(x)\prod_{i=2}^q
\int_{x-r}^{x+r} h_i(y)\,dy\,dx,
\]
and $I_2(r)$ is the integral over $I\setminus B_r(S)$.

\medskip
\noindent
For $x\notin B_r(S)$, the densities are regular, so
\[
\int_{x-r}^{x+r} h_i(y)\,dy \asymp r.
\]
Hence
\[
I_2(r)\asymp r^{q-1}.
\]

For $x$ close to $s$, we use the local behaviour
\[
\int_{x-r}^{x+r} h_i(y)\,dy
\asymp (|x-s|+r)^{1+\alpha_i(s)}.
\]
Thus
\[
I_1(r)
\asymp
\sum_{s\in S}
\int_{s-r}^{s+r}
|x-s|^{\alpha_1(s)} (|x-s|+r)^{\sum_{i=2}^q(1+\alpha_i(s))}\,dx.
\]
A direct scaling argument yields
\[
I_1(r)\asymp \sum_{s\in S} r^{q+\sum_{i=1}^q \alpha_i(s)}.
\]

Combining both contributions,
\[
I(r)\asymp r^{q-1} + r^{q(1+\overline{\alpha})}.
\]
The dominant term determines the scaling. If
\[
q(1+\overline{\alpha}) > q-1,
\]
i.e.\ $\overline{\alpha}>-1/q$, then $I(r)\asymp r^{q-1}$ and $C_q=1$. Otherwise,
$q(1+\overline{\alpha})\leq q-1$ and the singular contribution dominates:
\[
I(r)\asymp r^{q(1+\overline{\alpha})}.
\]
In either case, by definition of $C_q$,
\[
C_q = \lim_{r\to0}\frac{\log I(r)}{(q-1)\log r},
\]
which gives $C_q=1$ in the first case and $C_q=\dfrac{q(1+\overline{\alpha})}{q-1}$ in the second.
\end{proof}

\begin{remark}
The formula highlights a transition phenomenon: singularities only affect
$C_q$ when they are sufficiently strong (i.e.\ when $\overline{\alpha}<0$)
and when the number of trajectories $q$ is large enough. Otherwise, the
regular part of the measure dominates and $C_q=1$.
\end{remark}

\begin{example}
\label{ex:codim}
Let $I=[0,1]$.
\begin{enumerate}
\item Let $h_1(x)=h_2(x)=\dfrac{1}{\pi\sqrt{x(1-x)}}$ and $h_3(x)=1$. $h_1$ and $h_2$ are the invariant densities associated with the logistic map, while $h_3$ is that of an expanding linear map with integer slope. Each of $h_1,h_2$ has
exponents $\alpha_1(0)=\alpha_2(0)=\alpha_1(1)=\alpha_2(1)=-\tfrac{1}{2}$ and $\alpha_3\equiv0$,
giving $\overline{\alpha}=\tfrac{1}{3}(-\tfrac{1}{2}-\tfrac{1}{2}+0)=-\tfrac{1}{3}$.
The condition $q(1+\overline{\alpha})>q-1$ reads $3\cdot\tfrac{2}{3}>2$, which holds as an equality,
so the two exponents coincide and
\[
C_3(\mu_1,\mu_2,\mu_3)=1.
\]
\item Let $h_1(x)=h_2(x)=h_3(x)=h_4(x)=\dfrac{1}{\pi\sqrt{x(1-x)}}$ and $h_5(x)=1$. 
Here $\overline{\alpha}=\tfrac{1}{5}(4\cdot(-\tfrac{1}{2})+0)=-\tfrac{2}{5}$, and the condition
$q(1+\overline{\alpha})>q-1$ reads $5\cdot\tfrac{3}{5}>4$, i.e.\ $3>4$, which is false. Hence
\[
C_5(\mu_1,\ldots,\mu_5)=\frac{5(1-\tfrac{2}{5})}{4}=\frac{3}{4}.
\]
\end{enumerate}
\end{example}


\end{document}